\documentclass[12pt]{amsart}

\usepackage{etex, stmaryrd, epsfig, graphics, psfrag, latexsym,
  mathtools, mathrsfs,enumitem, longtable,booktabs,ifthen}
\usepackage{amsfonts, amsmath, amsthm, amssymb}
\usepackage[noadjust]{cite} \usepackage{xcolor}
\usepackage[all,cmtip]{xy}
\usepackage[lmargin=1in,rmargin=1in,tmargin=1in,bmargin=1in]{geometry}
\usepackage{tikz} \usetikzlibrary{matrix,arrows,shapes,calc}
\definecolor{darkblue}{rgb}{0,0,0.4} 
\usepackage{xr-hyper}
\usepackage[colorlinks=true, citecolor=darkblue, filecolor=darkblue, linkcolor=darkblue,urlcolor=darkblue]{hyperref} 
\usepackage[all]{hypcap}
\newcommand{\Sadey}{\begin{tikzpicture}[baseline=-\the\dimexpr\fontdimen22\textfont2\relax] \draw (0,0) circle (1ex); \node at (50:0.4ex) {.}; \node at (130:0.4ex) {.}; \draw (-35:0.75ex) arc (35:145:0.75ex);\end{tikzpicture}}
\newcommand{\Quesy}{\begin{tikzpicture}[baseline=-\the\dimexpr\fontdimen22\textfont2\relax] \draw (0,0) circle (1ex); 
\node at (0ex,-0.52ex) {.};
\draw[line width=0.7pt] plot[smooth] coordinates {(0ex,-0.2ex) (0ex,0ex) (0.2ex,0.25ex) (0.2ex,0.5ex) (0ex,0.6ex) (-0.25ex,0.55ex) (-0.3ex,0.4ex)};
\fill[black] (-0.3ex,0.4ex) circle (0.35pt);
\fill[black] (0ex,-0.2ex) circle (0.35pt);
\end{tikzpicture}}

\tikzset{f1a/.style={->, double}}
\tikzset{f2a/.style={->, dashed, double}}
\tikzset{Xa/.style={->, thick,densely dotted}}
\tikzset{Wa/.style={->, thick,dashed}}
\tikzset{diffa/.style={->, thick}}

\graphicspath{{draws/}{}}

\numberwithin{equation}{section}

\newtheorem{introthm}{Theorem}

\newtheorem{lem}{Lemma}[section]               
\newtheorem{lemma}[lem]{Lemma}               
\newtheorem{theorem}[lem]{Theorem}

\newtheorem{corollary}[lem]{Corollary}               

\newtheorem{proposition}[lem]{Proposition}
\newtheorem{citethm}[lem]{Theorem}
\theoremstyle{definition}

\newtheorem{definition}[lem]{Definition}

\theoremstyle{remark}

\newtheorem{remark}[lem]{Remark}

\newtheorem{example}[lem]{Example}

\newtheorem{convention}[lem]{Convention}

\numberwithin{figure}{section}
\numberwithin{table}{section}

\newcommand{\R}{\mathbb{R}}

\newcommand{\F}{\mathbb{F}}

\newcommand{\N}{\mathbb{N}}

\newcommand{\mc}{\mathcal}

\newcommand{\wt}{\widetilde}

\renewcommand{\emptyset}{\varnothing}

\newcommand{\alphas}{\boldsymbol{\alpha}}
\newcommand{\betas}{\boldsymbol{\beta}}

\newcommand{\from}{\nobreak\mskip2mu\mathpunct{}\nonscript
  \mkern-\thinmuskip{:}\penalty300\mskip6muplus1mu\relax}
\newcommand{\into}{\hookrightarrow}

\renewcommand{\th}{^{\text{th}}}

\newcommand{\SpinC}{\text{Spin}^{\text{c}}}

\newcommand{\HD}{\mathcal{H}}

\newcommand{\CF}{\mathit{CF}}
\newcommand{\HF}{\mathit{HF}}

\DeclareMathOperator{\Sym}{Sym}

\DeclareMathOperator{\Id}{Id}

\DeclareMathOperator{\Hom}{Hom}

\DeclareMathOperator{\Tor}{Tor}
\DeclareMathOperator{\Ext}{Ext}
\DeclareMathOperator{\RHom}{RHom}

\newcommand{\Kh}{\mathit{Kh}}
\newcommand{\rKh}{\widetilde{\Kh}}
\newcommand{\KhCx}{\mc{C}_{\mathit{Kh}}}
\newcommand{\rKhCx}{\widetilde{\mathcal{C}}_{\mathit{Kh}}}

\newcommand{\Filt}{\mathcal{F}}

\newcommand{\co}{\from}
\newcommand{\bdy}{\partial}
\newcommand{\RR}{\R}
\newcommand{\FF}{\F}

\newcommand{\pt}{\mathrm{pt}}

\newcommand{\NN}{\N}
\newcommand{\ZZ}{\mathbb{Z}}

\DeclareMathOperator{\image}{im}

\newcommand{\mathcenter}[1]{\vcenter{\hbox{$#1$}}}

\newcommand{\Ainf}{A_\infty}

\newcommand{\cM}{\mathcal{M}}

\tikzstyle{crossing}=[circle,fill=white,minimum height=6pt,inner sep=0pt, outer sep=0pt, style={transform shape=false}]

\definecolor{darkgreen}{rgb}{0,.3,0}
\definecolor{darkred}{rgb}{0.3,0,0}

\newcommand{\unroll}[1]{#1^{\mathrm{un}}}
\newcommand{\HFa}{\widehat{\HF}}
\newcommand{\CFa}{\widehat{\CF}}
\newcommand{\tHFa}{\underline{\widehat{\HF}}}
\newcommand{\tCFa}{\underline{\widehat{\CF}}}
\newcommand{\tbdy}{\underline{\bdy}}
\newcommand{\lra}{\longrightarrow}
\newcommand{\lla}{\longleftarrow}
\newcommand{\tors}{\mathit{tors}}
\newcommand{\othR}{S}

\begin{document}
\title{Khovanov homology detects split links}

\author{Robert Lipshitz}
\address{Department of Mathematics, University of Oregon\\
  Eugene, OR 97403}
\thanks{\texttt{RL was supported by NSF Grant DMS-1810893.}}
\email{\href{mailto:lipshitz@uoregon.edu}{lipshitz@uoregon.edu}}

\author{Sucharit Sarkar}
\address{Department of Mathematics, University of California, Los Angeles, CA 90095}
\thanks{\texttt{SS was supported by NSF Grant DMS-1905717.}}
\email{\href{mailto:sucharit@math.ucla.edu}{sucharit@math.ucla.edu}}

\date{\today}

\begin{abstract}
  Extending ideas of Hedden-Ni, we show that the module structure on
  Khovanov homology detects split links. We also prove an analogue for
  untwisted Heegaard Floer homology of the branched double
  cover. Technical results proved along the way include two
  interpretations of the module structure on untwisted Heegaard Floer
  homology in terms of twisted Heegaard Floer homology and the
  fact that the module structure on the reduced Khovanov complex of a
  link is well-defined up to quasi-isomorphism.
\end{abstract}

\maketitle

\tableofcontents

\section{Introduction}

Since the Jones polynomial and Khovanov homology are somewhat
mysterious invariants, there has been substantial interest in
understanding their geometric content. Much progress along these lines
has been finding detection results. Grigsby-Wehrli showed that the Khovanov
homology of nontrivial cables detects the
unknot~\cite{GW-kh-detects}. (See also~\cite{Hedden-Kh-2-cable}.)
Kronheimer-Mrowka showed that Khovanov homology itself detects the
unknot~\cite{KM-kh-detects}. So, by work of Hedden-Ni, Khovanov homology
also detects the 2-component unlink~\cite{HN-hf-small}. Hedden-Ni went
on to show that the module structure on Khovanov homology detects the
$n$-component unlink~\cite{HN-kh-detects}. Batson-Seed refined this to
show that Khovanov homology as a bi-graded abelian group detects the
unlink~\cite{BS-kh-splitting}. Recently, Baldwin-Sivek showed that
Khovanov homology detects the trefoils~\cite{BS-Kh-det-tref} and
Baldwin-Sivek-Xie showed that Khovanov homology detects the Hopf
links~\cite{BSX-kh-det-Hopf}. Even more recently, Xie-Zhang classified $n$-component links with Khovanov homology of dimension $2^n$~\cite{XZ-Kh-min-links}.

The detection problem for Heegaard Floer homology has also received
considerable attention. Ozsv\'ath-Szab\'o showed that knot Floer homology detects the genus
(and Heegaard Floer homology detects the Thurston
norm)~\cite{OSz-hf-genusbounds}, and hence the unknot. Ghiggini showed that knot
Floer homology detects the trefoils and figure 8 knot~\cite{Ghi-hf-fibred}. Ni
showed that knot Floer homology detects fibered knots in general and Heegaard
Floer homology detects 3-manifolds that fiber over the circle with fiber of
genus $>1$~\cite{Ni-hf-fibred,Ni-hf-fibred-closed}. Ai-Peters and Ai-Ni showed
that twisted Heegaard Floer homology detects fibered 3-manifolds with genus $1$
fibers~\cite{AP-hf-fibred,AN-hf-fibred}. Ni showed that Heegaard Floer homology
detects the Borromean knots~\cite{Ni-hf-homological}, and Hedden-Ni classified manifolds with small
Heegaard Floer ranks~\cite{HN-hf-small}. Building slightly on these results, Alishahi-Lipshitz
showed that bordered Heegaard Floer homology detects homologically essential
compressing disks, bordered-sutured Heegaard Floer homology detects
boundary-parallel tangles, and twisted Heegaard Floer homology detects homologically essential
2-spheres~\cite{AL-hf-incomp}. (This last detection theorem will be
used below.)

Indeed, all of the detection results for Khovanov homology come from
comparing Khovanov homology to some gauge-theoretic invariant, like
Heegaard Floer homology. This paper will be no exception. Extending
ideas of Hedden-Ni's, we will use the fact that the branched double
cover of a link $L$ is irreducible if and only if $L$ is prime and non-split
to show:
\begin{introthm}\label{thm:main-v1}
  Let $L$ be a 2-component link in $S^3$. Fix basepoints $p,q$ on the
  two components of $L$. Let $\rKh(L;\FF_2)$ be the reduced Khovanov
  homology of $L$ with respect to the basepoint $p$, viewed as an
  $\FF_2[X]/(X^2)$-module with respect to the basepoint $q$. Then,
  $\rKh(L;\FF_2)$ is a free module if and only if $L$ is a split link.

  More generally, for a link $L$ with $k$ components and basepoints
  $p,q$ on $L$, there is a $2$-sphere in $S^3\setminus L$ separating
  $p$ from $q$ if and only if $\rKh(L;\FF_2)$ is a free module over
  $\FF_2[X]/(X^2)$.
\end{introthm}

We give a refined version of Theorem~\ref{thm:main-v1}, and a version for
unreduced Khovanov homology, below, after recalling some algebra.

\begin{definition}\label{def:quasi-free}
  Let $C$ be a bounded chain complex over a ring $R$ or, more
  generally, an $\Ainf$-module over $R$. We say that $C$ is
  \emph{quasi-free} if $C$ is ($A_\infty$) quasi-isomorphic to a
  bounded chain complex of free $R$-modules.
\end{definition}

\begin{definition}\label{def:unrolled}
  Let $\FF_2[Y^{-1},Y]]$ denote the ring of Laurent series.
  Let $(C,\bdy_C)$ be a differential $\FF_2[X]/(X^2)$-module (e.g., a chain
  complex over $\FF_2[X]/(X^2)$). By the \emph{unrolling} of $C$ we mean the
  differential $\FF_2[Y^{-1},Y]]$-module
  $\unroll{C}=C\otimes_{\FF_2}\FF_2[Y^{-1},Y]]$ with differential
  \[
    \bdy(z\otimes Y^n)=\bdy_C(z)\otimes Y^{n}+zX\otimes Y^{n+1}.
  \]

  This is a completion of the total complex of the bicomplex
  \[
   \cdots\stackrel{X}{\lra} C\stackrel{X}{\lra} C\stackrel{X}{\lra}C\stackrel{X}{\lra}\cdots.
 \]

 More generally, if $C$ is a strictly unital $\Ainf$-module over $\FF_2[X]/(X^2)$ then
 the unrolled complex of $C$ is $C\otimes_{\FF_2} \FF_2[Y^{-1},Y]]$ with differential 
  \[
    \bdy(z\otimes Y^n)=
    \sum_{m\geq 0} \mu_{1+m}(z,\overbrace{X,\cdots,X}^m)\otimes Y^{n+m}.
  \]
  This is an honest differential module over $\FF_2[Y^{-1},Y]]$. (The
  notion of strict unitality is recalled in Definition~\ref{def:su}.)
  
 We will refer to the homology of $\unroll{C}$, $H_*(\unroll{C})$, as the
 \emph{unrolled homology} of $C$.
\end{definition}

\begin{introthm}\label{thm:2-comp-precise}
  Let $L$ be a link in $S^3$ and $p,q\in L$.
  Let $\rKhCx(L;\FF_2)$ be the reduced
  Khovanov complex with respect to $p$, which is a module over
  $\FF_2[X]/(X^2)$ via the basepoint $q$. Then, the following are
  equivalent:
  \begin{enumerate}[label=(\arabic*),leftmargin=*]
  \item\label{item:L-split} There is a $2$-sphere in $S^3\setminus L$ separating $p$ from
    $q$.
  \item\label{item:Kh-free} $\rKh(L;\FF_2)$ is a free module.
  \item\label{item:2-quasi-free} $\rKhCx(L;\FF_2)$ is quasi-free.
  \item\label{item:unroll-acyclic} $\unroll{\rKhCx(L;\FF_2)}$ is acyclic.
  \end{enumerate}
\end{introthm}

\begin{corollary}\label{cor:unreduced}
  Let $L$ be a link in $S^3$ and $p,q$ points in $L$. There is a
  $2$-sphere in $S^3\setminus L$ separating $p$ from $q$ if and only
  if $\Kh(L;\FF_2)$ is a free module over $\FF_2[W,X]/(W^2,X^2)$ where
  the action of $W$ corresponds to $p$ and the action of $X$
  corresponds to $q$.
\end{corollary}

\begin{remark}\label{rem:easy-parts}
  In Theorem~\ref{thm:2-comp-precise}, the implication~\ref{item:L-split}$\implies$\ref{item:Kh-free}
  is a result of Shumakovitch~\cite[Corollary 3.2.B]{Shu-kh-torsion};
  see Lemma~\ref{lem:split-is-free}.
  (This also follows from an argument in odd Khovanov
  homology~\cite[Proposition 1.8]{OSzR-kh-oddkhovanov}.)
  The implication~\ref{item:L-split}$\implies$\ref{item:2-quasi-free}
  is obvious, modulo knowing that the basepoint action is
  well-defined, up to quasi-isomorphism, on the reduced Khovanov
  complex. The
  implication~\ref{item:Kh-free}$\implies$\ref{item:unroll-acyclic}
  follows from an easy spectral sequence argument.  The
  implication~\ref{item:2-quasi-free}$\implies$\ref{item:unroll-acyclic}
  is Lemma~\ref{lem:qf-unroll-acyclic}, which again follows from an
  easy spectral sequence argument. Most of the work is in proving the
  implication~\ref{item:unroll-acyclic}$\implies$\ref{item:L-split},
  which uses the Ozsv\'ath-Szab\'o spectral
  sequence~\cite{OSz-hf-dcov}, a nontriviality result for twisted
  Heegaard Floer homology of Ozsv\'ath-Szab\'o and Hedden-Ni, and a
  computation of the $\Ainf$ module structure on $\HFa(Y)$ in terms of
  the twisted Floer homology. In particular, the restriction to
  characteristic $2$ is because of the corresponding restriction for
  the Ozsv\'ath-Szab\'o spectral sequence.
\end{remark}

As in Hedden-Ni's work, the key to proving
Theorem~\ref{thm:2-comp-precise} is tracking the module structure
through the Ozsv\'ath-Szab\'o spectral sequence
$\rKh(m(L))\Rightarrow \HFa(\Sigma(L))$. The Heegaard Floer homology
$\HFa(Y)$ is a module over the exterior algebra
$\Lambda^*(H_1(Y)/\tors)$~\cite{OSz-hf-3manifolds}. In the case
$Y=\Sigma(L)$, the pair of points $p,q\in L$ specifies an element
$X\in H_1(\Sigma(L))$ so, by restriction of scalars, $\HFa(\Sigma(L))$
is a module over $\FF_2[X]/(X^2)$. 

Proving Theorem~\ref{thm:2-comp-precise} requires working at
the chain level. As Hedden-Ni note, at the chain level, the action of
$X$ on $\CFa(\Sigma(L))$ is only associative up to homotopy. In fact,
$\CFa(\Sigma(L))$ is naturally an $\Ainf$-module over
$\FF_2[X]/(X^2)$; see Section~\ref{sec:HF-module}. By homological
perturbation theory, $\HFa(\Sigma(L))$ inherits the structure of an
$\Ainf$-module. Similarly, the action of $\FF_2[X]/(X^2)$ on $\rKhCx(L)$
induces an $\Ainf$-module structure on $\rKh(L)$.

We have the following Heegaard Floer analogue of
Theorem~\ref{thm:2-comp-precise}:
\begin{introthm}\label{thm:HF-detect-split}
  Let $L\subset S^3$ be a link and $p,q\in L$.  Consider the induced
  $\Ainf$-module structures on $\CFa(\Sigma(L);\FF_2)$ and
  $\HFa(\Sigma(L);\FF_2)$ over $\FF_2[X]/(X^2)$. We can also view
  $\HFa(\Sigma(L);\FF_2)$ as an ordinary module over $\FF_2[X]/(X^2)$, by
  forgetting the higher $\Ainf$ operations. Then, the following are
  equivalent:
  \begin{enumerate}[label=(\arabic*), ref=\arabic*,leftmargin=*]
  \item\label{item:HF-L-split} There is a $2$-sphere in
    $S^3\setminus L$ separating $p$ and $q$.
  \item\label{item:HF-free} $\HFa(\Sigma(L);\FF_2)$, viewed as an
    ordinary module over $\FF_2[X]/(X^2)$, is a free module.
  \item\label{item:HF-2-quasi-free} the $\Ainf$-module
    $\CFa(\Sigma(L);\FF_2)$ is quasi-free.
  \item\label{item:HF-unroll-cx-acyclic}
    $\unroll{\CFa(\Sigma(L);\FF_2)}$ is acyclic.
  \end{enumerate}
\end{introthm}
Note that, for Heegaard Floer homology with appropriate twisted
coefficients, some of these equivalences were essentially proved by
Hedden-Ni~\cite[Corollary 5.2]{HN-kh-detects}.

\begin{remark}
  This project stems from thinking about Eisermann's
  result~\cite{Eis-knot-Jones-ribbon} that the reduced Jones
  polynomial of a $2$-component ribbon link is divisible by
  $(q+q^{-1})$. Among his prescient comments, Eisermann~\cite[Section
  7.3]{Eis-knot-Jones-ribbon} notes that it is not true that the
  reduced Khovanov homology of such a ribbon link is divisible by
  the Khovanov homology of the unknot. We thought that perhaps,
  instead, the reduced Khovanov complex of a ribbon
  link might be quasi-free over $\FF_2[X]/(X^2)$, which would recover
  Eisermann's result after
  decategorification. Theorem~\ref{thm:2-comp-precise} shows that
  this is definitely not the case, at least in characteristic $2$. In
  fact, for Eisermann's example $L10^n_{36}$, a 2-component ribbon
  knot, the Khovanov complex is not quasi-free in any characteristic:
  the reduced Khovanov homology of $L10^n_{36}$, as computed by
  sKnotJob~\cite{Sch-kh-sknotjob}, is:
  \begin{center}
  \begin{tabular}{|c||c|c|c|c|c|c|c|c|c|c|c|}
\hline
\!$q$\! $\backslash$ \!$h$\! & $-5$ & $-4$ & $-3$ & $-2$ & $-1$ & $0$ & $1$ & $2$ & $3$ & $4$ & $5$ \\
\hline
\hline
$9$ & & & & & & & & & & & $\ZZ$ \\
\hline
$7$ & & & & & & & & & & $\ZZ$ & \\
\hline
$5$ & & & & & & & & & $\ZZ$ & & \\
\hline
$3$ & & & & & & & $\ZZ$ & $\ZZ^{2}$ & & & \\
\hline
$1$ & & & & & & $\ZZ^{2}$ & $\ZZ$ & & & & \\
\hline
$-1$ & & & & & $\ZZ$ & $\ZZ^{2}$ & & & & & \\
\hline
$-3$ & & & & $\ZZ^{2}$ & $\ZZ$ & & & & & & \\
\hline
$-5$ & & & $\ZZ$ & & & & & & & & \\
\hline
$-7$ & & $\ZZ$ & & & & & & & & & \\
\hline
$-9$ & $\ZZ$ & & & & & & & & & & \\
\hline
\end{tabular}
.
\end{center}
Considering the bi-gradings, this implies that the unrolled homology is nontrivial.
\end{remark}

\begin{remark}
  The restriction to $\FF_2$-coefficients in
  Theorem~\ref{thm:HF-detect-split} is presumably unnecessary. The
  additional work required to generalize
  Theorem~\ref{thm:HF-detect-split} to arbitrary field coefficients is
  adding signs to Section~\ref{sec:twist-to-Ainf}.
\end{remark}

\begin{remark}
  An analogue of Corollary~\ref{cor:unreduced} for link Floer
  homology was recently proved by Wang~\cite{Wang-hf-detects}.
\end{remark}

This paper is organized as follows. In Section~\ref{sec:alg-back} we
collect some algebraic definitions and
results. Section~\ref{sec:HF-module} recalls Heegaard Floer
homology with twisted coefficients and the ($\Ainf$)
$\Lambda^*(H_1(Y)/\tors)$-module structure on Heegaard Floer homology,
and relates them. The relations are in
Section~\ref{sec:twist-to-Ainf}; much of this works more generally for
complexes over $\FF_2[t^{-1},t]$ and $\Ainf$-modules over
$\FF_2[X]/(X^2)$, and may be of independent
interest. Section~\ref{sec:Kh-module} recalls the module
structure on the Khovanov complex and reduced Khovanov complex, and
proves these are invariants up to quasi-isomorphism. Finally,
Section~\ref{sec:detect} combines these ingredients to prove the
detection theorems.

\noindent
\emph{Acknowledgments.} We thank N.~Dunfield, M.~Hutchings, T.~Lidman, Y.~Ni, and
R.~Rouquier for helpful conversations. We also thank the referee for
further helpful corrections and suggestions.

\section{Algebraic background}\label{sec:alg-back}
Throughout this section, for convenience and because it suffices for
our application, we work in characteristic $2$. Many of the results
have easy extensions to arbitrary characteristic.

\subsection{Ungraded chain complexes}
The Heegaard Floer complexes are cyclically graded. Since the
homological algebra of cyclically graded chain complexes behaves
differently in some cases, we note some properties that hold
for ungraded chain complexes and, consequently, for cyclically graded
ones.

\begin{definition}
  Let $R$ be a ring. An \emph{ungraded chain complex} over
  $R$ or \emph{differential $R$-module} is an $R$-module $C$ and a
  homomorphism $\bdy\co C\to C$ with
  $\bdy^2=0$. The \emph{homology} $H(C,\bdy)$ of $(C,\bdy)$ is
  $\ker(\bdy)/\image(\bdy)$.
  
  Given ungraded chain complexes $(C,\bdy_C)$ and $(D,\bdy_D)$ over
  $R$, an $R$-module homomorphism $f\co C\to D$ is a \emph{chain map}
  if $\bdy_D\circ f = f\circ \bdy_C$. A chain map induces a map on
  homology. A chain map is a \emph{quasi-isomorphism} if the induced
  map on homology is an isomorphism.
\end{definition}

We will also be interested in ungraded $\Ainf$-modules:
\begin{definition}
  Let $R$ be an $\FF_2$-algebra. An \emph{ungraded $\Ainf$-module} over $R$ is
  an $\FF_2$-vector space $M$ together with maps
  \[
    \mu_{1+n}\co M\otimes R^{\otimes n}\to M
  \]
  satisfying
  \[
    \sum_{i+j=n}\mu_{1+i}(\mu_{1+j}(z,a_1,\dots,a_j),a_{j+1},\dots,a_n)+\sum_{i=1}^{n-1} \mu_{n}(z,a_1,\dots,a_{i-1},a_ia_{i+1},\dots,a_n)=0
  \]
  for each $n\geq 0$, $z\in M$, and $a_1,\dots,a_n\in R$.

  Given ungraded $\Ainf$-modules $(M,\mu^M)$ and $(N,\mu^N)$ over $R$,
  an \emph{$\Ainf$-module homomorphism} $f\co (M,\mu^M)\to (N,\mu^N)$ is a collection of
  $\FF_2$-vector space homomorphisms
  \[
    f_{1+n}\co M\otimes R^{\otimes n}\to N
  \]
  satisfying
  \begin{align*}
    \sum_{i+j=n}f_{1+i}(\mu^M_{1+j}(z,a_1,\dots,a_j),a_{j+1},\dots,a_n)
    &+\!\sum_{i+j=n}\!\mu^N_{1+i}(f_{1+j}(z,a_1,\dots,a_j),a_{j+1},\dots,a_n)\\
    &+\sum_{i=1}^{n-1} f_{n}(z,a_1,\dots,a_{i-1},a_ia_{i+1},\dots,a_n)=0    
  \end{align*}
  for each $n\geq 0$, $z\in M$, and $a_1,\dots,a_n\in R$.
  An $\Ainf$-module homomorphism $f$ is a \emph{quasi-isomorphism} if
  the map $f_1\co (M,\mu_1^M)\to (N,\mu_1^N)$ is a quasi-isomorphism.

  Given $\Ainf$ homomorphisms $f,g\co (M,\mu^M)\to(N,\mu^N)$, a
  \emph{homotopy} from $f$ to $g$ is a collection of $\FF_2$-vector
  space homomorphisms $k_{1+n}\co M\otimes R^{\otimes n}\to N$ so that
  for all $n$,
  \begin{multline*}
    \sum_{i+j=n}k_{1+i}(\mu^M_{1+j}(z,a_1,\dots,a_j),a_{j+1},\dots,a_n)
    +\!\sum_{i+j=n}\!\mu^N_{1+i}(k_{1+j}(z,a_1,\dots,a_j),a_{j+1},\dots,a_n)\\
    +\sum_{i=1}^{n-1} k_{n}(z,a_1,\dots,a_{i-1},a_ia_{i+1},\dots,a_n)=f_{1+n}+g_{1+n}.
  \end{multline*}

  Given $\Ainf$ homomorphisms $f\co (M,\mu^M)\to (N,\mu^N)$ and $g\co (N,\mu^N)\to (P,\mu^P)$, define
  $(g\circ f)\co M\to P$ by
  \[
    (g\circ f)_{1+n}=
    \sum_{i+j=n}g_{1+i}(f_{1+j}(z,a_1,\dots,a_j),a_{j+1},\dots,a_n).
  \]

  The \emph{identity homomorphism} of $M$ is defined by $\Id_1(x)=x$
  and $\Id_{1+n}=0$ for $n>0$.

  An $\Ainf$ homomorphism $f\co M\to N$ is a \emph{homotopy
    equivalence} if there is an $\Ainf$ homomorphism $g\co N\to M$ so
  that $f\circ g$ and $g\circ f$ are homotopic to the identity maps. 
\end{definition}
(Of course, these definitions generalize to the case that $R$ is an
$\Ainf$-algebra, but we will not need this generalization.)

The universal coefficient theorem holds in the ungraded setting:
\begin{lemma}\label{lem:univ-coeff}
  Let $R$ be a principal ideal domain, $(C,\bdy)$ an ungraded chain complex over $R$, and
  $M$ an $R$-module. Assume that $C$ is a projective $R$-module. Then,
  there is a natural short exact sequence
  \[
    0\to H(C,\bdy)\otimes_R M\to H(C\otimes_R M,\bdy)\to
    \Tor^1_R(H(C,\bdy),M)\to 0
  \]
  which splits (unnaturally).
\end{lemma}
\begin{proof}
  From $C$, construct an ordinary, bounded below, $\ZZ$-graded chain
  complex $\wt{C}$ by setting
  \[
    \wt{C}_n=
    \begin{cases}
      C & n\geq 0\\
      0 & n<0.
    \end{cases}
  \]
  and letting $\bdy_n\co \wt{C}_n\to \wt{C}_{n-1}$ be the map $\bdy$
  for all $n\geq 1$. Then, for any $i>0$, $H_i(\wt{C})\cong
  H(C)$. Applying the usual universal coefficient theorem for homology
  to $\wt{C}$ for any $i>0$ gives the result.
\end{proof}

\begin{proposition}\label{prop:PID-formal}
  Let $R$ be a principal ideal domain and let $C$ be a free chain
  complex over $R$. If $C$ is graded, assume that $C$ is finitely
  generated in each grading; if $C$ is ungraded, assume that $C$ is
  finitely generated. View the homology $H(C)$ as an honest
  $R$-module, i.e., with trivial higher operations $\mu_{1+n}$
  ($n>1$). Then, there is a quasi-isomorphism of $R$-modules
  $f\co C\to H(C)$.
\end{proposition}
\begin{proof}
  In the graded case, this is well-known; we observe that the proof
  also works for ungraded complexes $(C,\bdy)$. Let $K=\ker(\bdy)$. We
  claim that $C/K$ is a free module. Since $C/K$ is finitely
  generated, from the classification of modules over a PID it suffices
  to show that $C/K$ is torsion-free; but if $[\alpha]\in C/K$
  satisfies $r[\alpha]=0$ for some $r\in R$ then $r\alpha\in K$ so
  either $\alpha\in K$ or $r=0$.

  Hence, the short exact sequence $0\to K\to C\to C/K\to 0$
  splits. So, we can extend an ordered basis $x_1,\dots,x_k$ for $K$
  to an ordered basis $x_1,\dots,x_k,y_1,\dots,y_\ell$ for $C$. With respect to this basis, the matrix for $\bdy$ has the form
  \[
    \begin{bmatrix}
      0 & A\\
      0 & 0
    \end{bmatrix}
  \]
  where $A$ is some $k\times\ell$ matrix. By changing basis among the
  $x_i$ and $y_j$ we can assume $A$ is in Smith normal form. So,
  assume that $A$ has entries $r_1,\dots,r_j$ on the diagonal
  ($j=\min\{k,\ell\}$) and $0$s off the diagonal. Then,
  \[
    H(C)\cong R/(r_1)\oplus\cdots\oplus R/(r_j)
  \]
  and the homomorphism $C\to H(C)$ sending $x_i$ to $1\in R/(r_i)$ and
  $y_j$ to $0$ is a quasi-isomorphism.
\end{proof}

\begin{remark}
  In this paper, we make use of $\Ainf$-modules over $\FF_2[X]/(X^2)$ and chain
  complexes of honest modules over $\FF_2[t^{-1},t]$. One might wonder why
  $\Ainf$-modules over $\FF_2[t^{-1},t]$ do not also make an appearance. This is
  because of Proposition~\ref{prop:PID-formal}, which shows that no interesting
  $\Ainf$ operations over $\FF_2[t^{-1},t]$ arise. (In particular, there are no
  interesting $\Ainf$ operations on $\HFa(Y;\FF_2[t^{-1},t])$.)
\end{remark}

\subsection{Further notions for \texorpdfstring{$\Ainf$-modules}{A-infinity modules}}
In this section, we recall a few more definitions and results
regarding $\Ainf$-modules.

\begin{definition}\label{def:su}
  Let $R$ be a ring with unit $1$.  A (graded or ungraded) $\Ainf$-module
  $(M,\{\mu_{1+i}\})$ over $R$ is \emph{strictly unital} if:
  \begin{itemize}[leftmargin=*]
  \item $\mu_2(x,1)=x$ for all $x\in M$, and
  \item $\mu_{1+n}(x,a_1,\dots,a_n)=0$ if $n>1$ and some $a_i=1$.
  \end{itemize}
  Similarly, a morphism $\{f_{1+n}\co M\otimes R^{\otimes n}\to N\}$
  of strictly unital $\Ainf$-modules is \emph{strictly unital} if
  $
    f_{1+n}(m,a_1,\dots,a_n)=0
  $
  if some $a_i=1$.
\end{definition}

\begin{convention}\label{conv:unital}
  Throughout this paper, all $\Ainf$-modules and maps are strictly unital.
\end{convention}

\begin{example}
  A strictly unital $\Ainf$-module over $\FF_2[X]/(X^2)$ is determined
  by the operations $\mu_{1+n}(\cdot,X,\dots,X)$.
\end{example}

There are several advantages of working with $\Ainf$-modules; we
highlight two (related) ones. First, $\Ainf$-module structures
transfer nicely under maps; results of this kind for $\Ainf$ objects
are often called \emph{homological perturbation theory}:
\begin{proposition}\label{prop:hpt}
  Let $R$ be an $\FF_2$-algebra and $(M,\mu^M)$ an $\Ainf$-module over
  $R$. Let $(N,\mu^N_1)$ be a chain complex over $\FF_2$ and
  $f_1\co (M,\mu^M_1)\to (N,\mu^N_1)$ a homotopy equivalence of chain
  complexes over $\FF_2$. Then, there is an $\Ainf$ structure
  $\{\mu^N_{1+n}\}$ on $N$ extending $\mu^N_1$ and an $\Ainf$ homotopy
  equivalence $f\co (M,\mu^M)\to(N,\mu^N)$ extending $f_1$.

  The corresponding statement also holds for $\Ainf$
  $(R,S)$-bimodules.
\end{proposition}
See, e.g., Keller's survey~\cite[Section 4.3]{Keller-other-intro},
or~\cite[Lemma 9.6]{LOT-hf-compute}. In particular, the former
reference has a nice description of the history of such results, and
the latter does not rely on gradings.

Second, for differential modules or chain complexes of modules, there
is an important distinction between homotopy equivalence and
quasi-isomorphism. This distinction does not exist for $\Ainf$-modules:
\begin{proposition}\label{prop:qi-is-he}
  Let $R$ be an algebra over $\FF_2$, $M$ and $N$ $\Ainf$ modules over
  $R$, and $f\co M\to N$ a quasi-isomorphism. Then, $f$ is a homotopy
  equivalence. Further, two ordinary differential modules $M$, $N$
  over $R$ are $\Ainf$ quasi-isomorphic (or homotopy equivalent) if
  and only if $M$ and $N$ are quasi-isomorphic in the usual sense.
\end{proposition}
See, e.g., Keller's paper~\cite[Section 4.3]{Keller-other-intro},
or~\cite[Proposition 2.4.1]{LOT-hf-bimodules}. (Again, the latter
reference does not rely on gradings.) The point is that the map from
the bar resolution of $M$ to $M$ is an $\Ainf$ homotopy equivalence,
and the bar resolution is a projective module. Hence, the
quasi-isomorphism to the bar resolution is invertible up to homotopy,
and hence any quasi-isomorphism is invertible up to homotopy.

\subsection{The unrolled homology}\label{sec:unrolled}
Recall that given an $\Ainf$-module $C$ over $\FF_2[X]/(X^2)$, in
Definition~\ref{def:unrolled} we defined the unrolled complex
$\unroll{C}$ of $C$.

\begin{lemma}\label{lem:X-hom-well-def-ungr}
  Let $(C,\{\mu_{1+n}^C\})$ and $(D,\{\mu_{1+n}^D\})$ be finitely
  generated, graded or ungraded $\Ainf$-modules over $\othR=\FF_2[X]/(X^2)$. A
  homomorphism of $\Ainf$-modules $f\co C\to D$ induces a homomorphism
  $F\co \unroll{C}\to\unroll{D}$, and if $f$ is a quasi-isomorphism
  then so is $F$.
\end{lemma}
\begin{proof}
  Given a collection of maps
  $f_{1+n}\co C\otimes_{\FF_2} \othR^{\otimes n}\to D$ define a map
  \[
    F\co \unroll{C}\to\unroll{D}
  \]
  by 
  \[
    F(z\otimes Y^n)= \sum_{m\geq 0}
    f_{1+m}(z,\overbrace{X,\cdots,X}^m)\otimes Y^{n+m}.
  \]

  It is immediate from the construction that:
  \begin{itemize}[leftmargin=*]
  \item If $f$ is the identity map (i.e., $f_1=\Id$ and $f_n=0$ for
    $n>1$) then the induced map $F$ is also the identity map.
  \item The map $F$ associated to a collection of maps $f=\{f_{1+n}\}$ is
    well-defined. (In particular, this uses the fact that we have
    completed with respect to $Y$.)
  \item The map $F$ associated to a collection of maps $f=\{f_{1+n}\}$ is an
    $\FF_2[Y^{-1},Y]]$-module homomorphism.
  \item If $f=\{f_{1+n}\}$ is an $\Ainf$-module homomorphism then $F$ is a chain
    map. (In fact, $F$ is a chain map if and only if $f$ is an $\Ainf$
    module homomorphism.)
  \item If $k$ is a homotopy between $\Ainf$-module homomorphisms $f$
    and $g$ then the induced map $K$ is a chain homotopy between $F$
    and $G$.
  \item The map associated to $g\circ f$ is the composition of the
    maps $G$ associated to $g$ and $F$ associated to $f$.
  \end{itemize}
  It follows that homotopy equivalent $\Ainf$-modules have homotopy
  equivalent unrolled complexes. Since by
  Proposition~\ref{prop:qi-is-he},
  quasi-isomorphism and homotopy
  equivalence agree for $\Ainf$-modules (over an algebra over a
  field), this proves the result.
\end{proof}

\begin{lemma}\label{lem:qf-unroll-acyclic}
  Let $C$ be a (graded or ungraded) chain complex over
  $\FF_2[X]/(X^2)$, not necessarily free.  If $C$ is quasi-free then
  $\unroll{C}$ is acyclic.
\end{lemma}
\begin{proof}
  By Lemma~\ref{lem:X-hom-well-def-ungr} it suffices to prove the
  result when $C$ is a finite-dimensional free module (with a
  differential).

  As a warm-up, we start with the graded case when $X$ has grading
  $0$. Consider the spectral sequence associated to the vertical
  filtration on $\unroll{C_*}$, where the $d^0$-differential is multiplication by
  $X$. Since $C_*$ is free, this $d^0$-differential is exact. Hence,
  for this spectral sequence, $E^1=0$. Since $C_*$ is bounded, this
  implies that $H_*(\unroll{C_*})=0$, as well. This proves the result.

  Essentially the same argument works in the general (ungraded)
  case. Choose an ordered basis $[e_1,\dots,e_N,f_1,\dots,f_N]$ for
  $C$ over $\FF_2$, where $Xe_i=f_i$. Write the differential on $C$ as
  a block matrix
  $\left(\begin{smallmatrix}
      A & D\\
      B & E
    \end{smallmatrix}\right)$
  where each block is $N\times N$. Since $\bdy(f_i)=X\bdy(e_i)$, we have $D=0$ and $A=E$, so the differential actually has the form
  $\left(\begin{smallmatrix}
      A & 0\\
      B & A
    \end{smallmatrix}\right)$.

  The complex $\unroll{C}$ is a vector space over $\FF_2[Y^{-1},Y]]$ with ordered basis $[e_1\otimes 1,\dots,e_N\otimes 1,f_1\otimes 1,\dots,f_N\otimes 1]$.
  The differential on $\unroll{C}$ has the form
  \[
    \begin{pmatrix}
      A & 0\\
      B+IY & A
    \end{pmatrix}
  \]
  where $I$ denotes the $N\times N$ identity matrix.  Since
  $\bdy^2=0$, the differential on $\unroll{C}$ has rank at most $N$,
  so since
  \[
    \det(B+IY)=Y^N+\text{lower order terms}\neq 0,
  \]
  $(B+IY)$ is invertible, the differential on $\unroll{C}$ has rank
  equal $N$. Hence, since $\FF_2[Y^{-1},Y]]$ is a field, $\unroll{C}$ is
  acyclic, as claimed.
\end{proof}

The spectral sequence in the (graded case of the) proof of
Lemma~\ref{lem:qf-unroll-acyclic} is only well-behaved under
restrictive hypotheses: for unbounded chain complexes,
convergence becomes a problem. (Consider, for
example, the chain complex
$0\leftarrow \FF_2[X]/(X^2)\stackrel{X}{\lla} \FF_2[X]/(X^2)\stackrel{X}{\lla}\cdots$.)
On the other hand, because we have completed with respect to $Y$, the
\emph{horizontal filtration} of $\unroll{C_*}$, by the power of $Y$, induces
a spectral sequence that is well-behaved even for $C_*$ unbounded or
ungraded. For this spectral sequence, the $d^0$-differential is the
differential on $C_*$, the $d^1$-differential is the action of $X$ on
the homology of $C_*$, and the higher differentials are induced from
the $\Ainf$ operations on the homology of $C_*$.

\begin{remark}
  In the language of bordered Floer theory~\cite[Section
  8]{LOT-hf-HomPair}, there is a rank $1$ type $\mathit{DD}$ bimodule over
  $\FF_2[X]/(X^2)$ and $\FF_2[Y]$ defined by $P=\langle\iota\rangle$ and
  $\delta^1(\iota)=(X\otimes Y)\otimes\iota$. The bimodule $P$
  witnesses the Koszul duality between $\FF_2[X]/(X^2)$ and
  $\FF_2[Y]$. The unrolled complex is obtained by taking the box
  tensor product with $P$, modulifying the result, and extending
  scalars from $\FF_2[[Y]]$ to $\FF_2[Y^{-1},Y]]$. The appearance of
  power series in $Y$ relates to operational boundedness
  (cf.~\cite[Section 9]{LOT-hf-diags}).
\end{remark}

\section{Two views of the module structure on Heegaard Floer homology}\label{sec:HF-module}
\subsection{Geometry: holomorphic curves with point constraints}\label{sec:holomorphic}
Fix a commutative ring $k$.

Let $Y$ be a closed, oriented $3$-manifold and let
$\HD=(\Sigma,\alphas,\betas,z)$ be a weakly admissible pointed
Heegaard diagram for $Y$. Given an abelian group $G$, a
\emph{$G$-valued additive assignment} is a function
$A\co \pi_2(x,y)\to G$ for each pair of points
$x,y\in T_\alpha\cap T_\beta$ so that for all
$w,x,y\in T_\alpha\cap T_\beta$, $\phi\in\pi_2(w,x)$, and
$\psi\in\pi_2(x,y)$, $A(\phi*\psi)=A(\phi)+A(\psi)$. Given a
$G$-valued additive assignment $A$, there is an associated twisted
Floer complex with coefficients in the group ring $\FF_2[G]$,
\[
  \tCFa(Y;\FF_2[G]_A)=\tCFa(\HD;\FF_2[G]_A)=\bigoplus_{x\in T_\alpha\cap T_\beta}\FF_2[G],
\]
with differential
\[
  \underline{\bdy}(x)=
  \sum_{y\in T_\alpha\cap T_\beta}
  \sum_{\substack{\phi\in\pi_2(x,y)\\\mu(\phi)=1,\ n_z(\phi)=0}}
  \bigl(\#\cM^\phi(x,y)\bigr)t^{A(\phi)}y.
\]
Here, we are writing elements of $\FF_2[G]$ as linear combinations
$\sum n_it^{g_i}$ with $n_i\in \FF_2$ and $g_i\in G$, and $\cM^\phi(x,y)$
denotes the moduli space of holomorphic Whitney disks connecting $x$
to $y$ in the homotopy class $\phi$ (modulo the action of $\RR$ on the
source), with respect to a sufficiently generic family of almost
complex structures. It turns out that there is a universal,
\emph{totally twisted coefficient} Floer homology
$\tCFa(Y;\FF_2[H_2(Y)]_A)$, where $A$ is any $H_2(Y)$-valued additive assignment which is
bijective on $\{\phi\in \pi_2(x,x)\mid n_z(\phi)=0\}$, and any other
twisted Floer complex is obtained from the totally twisted coefficient
Floer complex by extension of scalars. (In particular,
Ozsv\'ath-Szab\'o originally defined Heegaard Floer homology with
twisted coefficients via the totally twisted Floer complex and
extension of scalars~\cite[Section 8]{OSz-hf-applications}.)

Recall that each homotopy class $\phi\in\pi_2(x,y)$ is represented by
a cellular $2$-chain in $(\Sigma,\alphas\cup\betas)$, called its
\emph{domain} $D(\phi)$. Let $\bdy_\alpha D(\phi)$ be the part of
$\bdy D(\phi)$ lying in the $\alpha$-circles. Fix an embedded,
oriented $1$-manifold $\zeta\subset \Sigma$ which intersects $\alphas$
transversely and is disjoint from $\alphas\cap\betas$. There is a
corresponding $\ZZ$-valued additive assignment
\[
  \phi\mapsto \zeta\cdot \bdy_\alpha D(\phi),
\]
the algebraic intersection number of $\zeta$ with
$\bdy_\alpha D(\phi)$. This additive assignment gives a twisted
coefficient complex $\tCFa(\HD;\FF_2[t^{-1},t]_{\zeta})$ with differential
\begin{equation}\label{eq:t-bdy}
  \underline{\bdy}(x)=
  \sum_{y\in T_\alpha\cap T_\beta}
  \sum_{\substack{\phi\in\pi_2(x,y)\\\mu(\phi)=1,\ n_z(\phi)=0}}
  \bigl(\#\cM^\phi(x,y)\bigr)t^{\zeta\cdot \bdy_\alpha D(\phi)}y.
\end{equation}
It is not hard to show that, up to quasi-isomorphism, the complex
$\tCFa(\HD;\FF_2[t^{-1},t]_{\zeta})$ depends on $\zeta$ only through the
homology class $[\zeta]\in H_1(Y)/\tors=\Hom(H_2(Y),\ZZ)$ it represents.

Of course, there is also an untwisted Heegaard Floer homology group
\[
  \CFa(Y;\FF_2)=\CFa(\HD;\FF_2)=\bigoplus_{x\in T_\alpha\cap T_\beta}\FF_2
\]
with differential
\begin{equation}\label{eq:un-t-bdy}
  \bdy(x)=
  \sum_{y\in T_\alpha\cap T_\beta}
  \sum_{\substack{\phi\in\pi_2(x,y)\\\mu(\phi)=1,\ n_z(\phi)=0}}
  \bigl(\#\cM^\phi(x,y)\bigr)y.
\end{equation}
As Ozsv\'ath-Szab\'o noted~\cite[Section 4.2.5]{OSz-hf-3manifolds}, the untwisted
Heegaard Floer complex $\CFa(Y;\FF_2)$ inherits an action of
$H_1(Y)/\tors$ via the formula
\begin{equation}\label{eq:H1-act-1}
  \zeta\cdot x=
  \sum_{y\in T_\alpha\cap T_\beta}
  \sum_{\substack{\phi\in\pi_2(x,y)\\\mu(\phi)=1,\ n_z(\phi)=0}}
  \bigl(\#\cM^\phi(x,y)\bigr)\bigl(\zeta\cdot \bdy_\alpha D(\phi)\bigr)y.
\end{equation}
(The action of $\zeta$ lowers the Maslov grading by $1$.)
As they show, at the level of homology this endows $\HFa(Y)$ with the
structure of a module over the exterior algebra
$\Lambda^*(H_1(Y)/\tors)$.  (There is a tiny but relevant omission in
Ozsv\'ath-Szab\'o's argument~\cite[Proof of Proposition
4.17]{OSz-hf-3manifolds}: they dropped the homotopy term which is
discussed below. See also~\cite[Proof of Proposition
8.6]{Lip-hf-cylindrical}.)

This statement can be refined slightly to make $\CFa(\HD)$, and hence
$\HFa(\HD)$, into an $\Ainf$-module over
$\Lambda^*(H_1(Y)/\tors)$. This is a special case of the quantum cap
product in Floer (co)homology, as sketched, say, by Seidel~\cite[Section
  8l]{Seidel-top-book} or Perutz~\cite[Section
  3.9]{Perutz-top-matching2}. 
Rather than describe the general case, we will focus on the action
by a single element of $H_1(Y)$, where tracking perturbations is less
cumbersome; this is sufficient for our applications.

So, fix an oriented multicurve $\zeta\subset \Sigma$ representing an
element of $H_1(Y)$, such that $\zeta\pitchfork\alphas$ and
$\zeta\cap\alphas\cap\betas=\emptyset$. Let $A_1=\zeta\cap\alphas$.
The orientations of $\zeta$ and $\Sigma$ induce a coorientation of
$\zeta$; let $A_i\subset\alphas$ be a small pushoff of $A_1$ so that
each point of $A_{i+1}$ is in the negative direction of
the coorientation of $\zeta$ from the corresponding point of $A_i$.

There are corresponding subsets
\begin{align*}
  C_{i}&=\{(x_1,\dots,x_g)\in T_\alpha\mid x_k\in A_i\text{ for some }k\}\\
  C_{i,j}&=\{(x_1,\dots,x_g)\in T_\alpha\mid x_k\in A_i,\ x_\ell\in
              A_j\text{ for some }k\neq \ell\}.
\end{align*}
The sets $C_i$ and $C_{i,j}$ are finite unions of submanifolds of
$T_\alpha$, of codimension 1 and 2, respectively.

Given integers $i_1,\cdots,i_k$, consider the moduli space
\begin{equation}\label{eq:C-mod-1}
  \cM^\phi(x,y;C_{i_1},\dots,C_{i_k})
\end{equation}
of holomorphic Whitney disks
\[
  u\co ([0,1]\times\RR, \{1\}\times\RR,\{0\}\times\RR)\to (\Sym^g(\Sigma),T_\alpha,T_\beta)
\]
together with points $(1,t_1),\dots,(1,t_k)\in\{1\}\times\RR$ with $t_1<\cdots<t_k$ with
$u(1,t_j)\in C_{i_j}$. There is also a moduli space
\begin{equation}\label{eq:C-mod-2}
  \cM^\phi(x,y;C_{i_1},\dots,C_{i_{\ell-1}},C_{i_\ell,i_{\ell+1}},C_{i_{\ell+2}},\dots,C_{i_k})
\end{equation}
defined similarly except with $u(1,t_\ell)\in C_{i_\ell,i_{\ell+1}}$.

Choose $\zeta$ so that for every disk $u$ with Maslov index $1$,
$C_1\pitchfork u|_{\{1\}\times\RR}$. (This is possible since there are finitely many
disks $u$ with Maslov index $1$.) Let $U$ be a neighborhood of
$A_1=\zeta\cap\alphas$ small enough that for all Maslov index $1$
disks $u$ and all  $a\in U$,
\[
 \{(x_1,\dots,x_g)\in T_\alpha\mid x_k=a\text{ for some }k\}\pitchfork u|_{\{1\}\times\RR}.
\]
Choose
the perturbations $A_i$ to be entirely contained in $U$.  Then, these
perturbations have the following two properties:
\begin{enumerate}[label=(M-\arabic*),leftmargin=*]
\item The moduli spaces in Equations~\eqref{eq:C-mod-1}
  and~\eqref{eq:C-mod-2} are transversely cut out.
\item\label{item:close-1} The moduli spaces
  $\cM^\phi(x,y;C_{1},\dots,C_{k})$ and
  $\cM^\phi(x,y;C_{i+1},\dots,C_{i+k})$ are identified for all
  $i$. 
\end{enumerate}

Now, define the operation
\[
  \mu_{1+n}\co \CFa(\HD)\otimes \FF_2[X]/(X^2)^{\otimes n}\to \CFa(\HD)
\]
by
\begin{equation}\label{eq:Ainf-holo}
  \mu_{1+n}(x,X,\dots,X)= \sum_{y\in T_\alpha\cap T_\beta}
  \sum_{\substack{\phi\in\pi_2(x,y)\\\mu(\phi)=1,\ n_z(\phi)=0}}
  \bigl(\#\cM^\phi(x,y;C_1,\dots,C_n)\bigr)y.
\end{equation}
Define the operation $\mu_1$ to be the differential on
$\CFa(\HD)$. Observe that the operation $\mu_2$ is the restriction of
the $H_1(Y)/\tors$ action.

\begin{lemma}\label{HF-ainf-well-defd}
  The operations $\mu_{1+n}$ satisfy the $\Ainf$ relations, so
  $\CFa(\HD)$ inherits the structure of an $\Ainf$-module.
\end{lemma}
\begin{proof}
  Consider the boundary of the moduli space
  \[
    \bigcup_{\substack{\phi\in\pi_2(x,y)\\\mu(\phi)=2,\ n_z(\phi)=0}}
    \cM^\phi(x,y;C_1,\dots,C_n).
  \]
  This moduli space has two kinds of boundary points: points in
  \[
    \bigcup_{w\in T_\alpha\cap T_\beta}
    \bigcup_{\substack{\phi_1\in\pi_2(x,w),\phi_2\in\pi_2(w,y)\\\mu(\phi_i)=1,\ n_z(\phi_i)=0}}
    \cM^{\phi_1}(x,w;C_1,\dots,C_k)\times \cM^{\phi_2}(w,y;C_{k+1},\dots,C_n)
  \]
  and points in
  \[
    \bigcup_{\substack{\phi\in\pi_2(x,y)\\\mu(\phi)=2,\ n_z(\phi)=0}}
    \cM^\phi(x,y;C_1,\dots,C_{m,m+1},\dots,C_n).
  \]
  By Condition~\ref{item:close-1}, points of the first kind correspond
  to the term
  \[
    \mu_{n-k+2}(\mu_{k+1}(x,X,\dots,X),X,\dots,X)
  \]
  in the $\Ainf$ relation. Points in the second kind of terms come in pairs: An
  element $u\in\cM^\phi(x,y;C_1,\dots,C_{m,m+1},\dots,C_n)$ with
  $u(1,t_m)=v\in C_{m,m+1}$ with $v_k\in A_m\cap \alpha_k$ and
  $v_\ell \in A_{m+1}\cap \alpha_\ell$ is paired with a nearby
  $u'\in \cM^\phi(x,y;C_1,\dots,C_{m,m+1},\dots,C_n)$ with
  $u'(1,t_m)=v'\in C_{m,m+1}$ with $v'_k\in A_{m+1}\cap \alpha_k$ and
  $v'_\ell \in A_{m}\cap \alpha_\ell$, using the condition \ref{item:close-1} on
  these types of moduli spaces.
  This proves the result.
\end{proof}

We will show next that the counts of the moduli spaces
$\cM^\phi(x,y;C_{i_1},\dots,C_{i_k})$ are completely determined by the
moduli spaces $\cM^\phi(x,y)$ and the homotopy classes $\phi$. (A key
point is that the sets $C_i\subset T_\alpha$ have codimension $1$.)
As a first step, given a curve $u\in\cM^\phi(x,y)$, let $N_k(u)$ be
the number of tuples $t_1<t_2<\cdots<t_k$ so that $u(1,t_i)\in C_i$
or, equivalently, so that $u(1,t_i)\cap A_i\neq\emptyset$. Then
\[
  \mu_{1+n}(x,X,\dots,X)= \sum_{y\in T_\alpha\cap T_\beta}
  \sum_{\substack{\phi\in\pi_2(x,y)\\\mu(\phi)=1,\ n_z(\phi)=0}}
  \sum_{u\in\cM^\phi(x,y)} N_n(u)y.
\]

Recall that given integers $m,n$ with $n\geq 0$ there is an
integer ${m\choose n}=m(m-1)\cdots(m-n+1)/n!\in\ZZ$, which reduces to
an element ${m\choose n}\in\FF_2$.
\begin{lemma}\label{lem:HF-ainf-formula}
  The $\Ainf$ operation $\mu_{1+n}$ from Formula~\eqref{eq:Ainf-holo}
  is given by
  \[
    \mu_{1+n}(x,X,\dots,X)=\sum_{y\in T_\alpha\cap T_\beta}
    \sum_{\substack{\phi\in\pi_2(x,y)\\\mu(\phi)=1,\ n_z(\phi)=0}}
    {\zeta\cdot\bdy_\alpha D(\phi)\choose n}\bigl(\#\cM^\phi(x,y)\bigr)y.
  \]
  where $\zeta\cdot\bdy_\alpha D(\phi)$ denotes the algebraic
  intersection number of $\zeta$ with the part of $\bdy D(\phi)$ lying
  in $\alphas$.
\end{lemma}
\begin{proof}
  Consider a holomorphic curve $u\in\cM^\phi(x,y)$ so that
    $(u|_{\{1\}\times\RR})^{-1}(C_1)$
  consists of $a+b$ points, $a$ of which are positive and $b$ of which
  are negative. (In other words, the boundary of $u$, viewed as a
  smooth 1-chain in $\Sigma$, intersects $\zeta$ $a$ times positively
  and $b$ times negatively.) We claim that
  $N_{n}(u)\equiv {a-b\choose n}\pmod 2$. In particular, this implies that
  $N_n(u)$ depends only on the intersection number $a-b$ of
  $\bdy_\alpha D(\phi)$ and $\zeta$:
  \[
    N_n(u)\equiv {a-b\choose n}= {\zeta\cdot\bdy_\alpha D(\phi)\choose n}\pmod 2.
  \]
  This, then, will immediately imply the result.
  
  To see that $N_{n}(u)\equiv {a-b\choose n}\pmod 2$, suppose that
  \[
    (u|_{\{1\}\times\RR})^{-1}(C_1)=\{(1,t_1),(1,t_2),\dots,(1,t_\ell)\}
  \]
  where $t_1<\cdots<t_\ell$. (In the notation of the previous
  paragraph, $\ell=a+b$.) Let $s_i\in\{\pm 1\}$ be the sign of
  $(1,t_i)\in (u|_{\{1\}\times\RR})^{-1}(C_1)$, with respect to the coorientation
  of $\zeta$ and the orientation of $\{1\}\times\RR$.  The inverse
  function theorem implies that there are small, positive real numbers
  $\epsilon_{2,1},\dots,\epsilon_{n,\ell}$ so that
  \begin{align*}
    (u|_{\{1\}\times\RR})^{-1}(C_2)&=\bigl\{(1,t_1-s_1\epsilon_{2,1}),(1,t_2-s_2\epsilon_{2,2}),\dots,(1,t_\ell-s_\ell\epsilon_{2,\ell})\bigr\}\\
    (u|_{\{1\}\times\RR})^{-1}(C_3)&=\bigl\{\bigl(1,t_1-s_1(\epsilon_{2,1}+\epsilon_{3,1})\bigr),\bigl(1,t_2-s_2(\epsilon_{2,2}+\epsilon_{3,2})\bigr),\dots,\bigl(1,t_\ell-s_\ell(\epsilon_{2,\ell}+\epsilon_{3,\ell})\bigr)\bigr\}\\
    (u|_{\{1\}\times\RR})^{-1}(C_4)&=\bigl\{\bigl(1,t_1-s_1(\epsilon_{2,1}+\epsilon_{3,1}+\epsilon_{4,1})\bigr),\dots,\bigl(1,t_\ell-s_\ell(\epsilon_{2,\ell}+\epsilon_{3,\ell}+\epsilon_{4,\ell})\bigr)\bigr\}
  \end{align*}
  and so on. In particular, suppose $j<k$. The preimage $(1,t_i)$ of
  $C_1$ gives preimages of $C_j$ and $C_k$ that occur in order if
  $s_i<0$ and out of order if $s_i>0$. It follows that $N_n(u)$ is the
  number $N_n(a,b)$ of ways of choosing $n$ points among the $a+b$ intersection
  points, possibly with repetitions, subject to the restriction that
  positive intersection points cannot be repeated
  
  It remains to prove that $N_{n}(a,b)\equiv {a-b\choose n}\pmod
  2$. The number $N_n(a,b)$ is the coefficient of $s^{n}$
  in $(1+s)^a(1+s+s^2+\cdots)^b$: the $a$ $(1+s)$ factors represent
  the $a$ positive intersections, which can be chosen $0$ or $1$
  times, and the the $b$ $(1+s+s^2+\cdots)$ factors represent the $b$
  negative intersections, which can be chosen any number of times.
  Since $1+s+s^2+\cdots\equiv 1-s+s^2-\cdots\equiv(1+s)^{-1}\pmod 2$, this
  equals
  $(1+s)^a(1+s)^{-b}=(1+s)^{a-b}=\sum_{n\in\NN} {a-b\choose n}s^n$,
  and so $N_{n}(a,b)\equiv{a-b\choose n}\pmod 2$, as desired.
  %
  %
\end{proof}

\begin{theorem}\label{thm:CFa-Ainf-well-defd}
  Up to quasi-isomorphism, the $\Ainf$-modules $\CFa(\HD)$ and
  $\HFa(\HD)$ over the ring $\FF_2[X]/(X^2)$ are independent of the multi-curve
  $\zeta$ representing $[\zeta]\in H_1(Y)/\tors$, the perturbations,
  the Heegaard diagram, and the almost complex structure in their
  construction.
\end{theorem}
\begin{proof}
  This is a simple adaptation of the usual invariance proof for
  Heegaard Floer homology~\cite{OSz-hf-3manifolds}, and is left to the
  reader. The result also follows from invariance of
  $\tHFa(Y;\ZZ[t^{-1},t]_\zeta)$ and Theorem~\ref{thm:twist-to-X}
  below (whose proof does not depend on this theorem).
\end{proof}

\begin{remark}
  We have suppressed $\SpinC$-structures from the discussion
  above. All of the complexes decompose as direct sums over the
  $\SpinC$-structures on $Y$, and the $\Ainf$ action respects this
  decomposition.
\end{remark}

\subsection{Algebra: twisted coefficients and Koszul duality}\label{sec:twist-to-Ainf}
Michael Hutchings pointed out to the first author around 2004 that one
can recover the $H_1/\tors$-action on $\HFa(Y)$ from $\tCFa(Y)$. As he
may also have explained, this extends to the $\Ainf$-module
structure. In this section, we give two formulations of this
construction.

To motivate the first formulation, consider the relations between
Equations~(\ref{eq:t-bdy}),~(\ref{eq:un-t-bdy}),
and~(\ref{eq:H1-act-1}): the operation $\bdy$ is obtained from $\tbdy$
by setting $t=1$, while the operation $\zeta\cdot$ is obtained from
$\tbdy$ by differentiating once with respect to $t$ and then setting
$t=1$. (Compare the operators $\Phi_w$ and
$\Psi_z$~\cite{Sar-action,Zemke-hf-bp-moving}.) Of course, the second derivative
vanishes in characteristic $2$, but the right generalization of this
relation was introduced by Hasse and
Schmidt~\cite{SH-other-derivatives}:

\begin{definition}
  Let $k$ be a commutative ring with unit and let $k[t^{-1},t]$ be
  the ring of Laurent polynomials over $k$. Given $m,n\in\ZZ$, $n\geq 0$,
  the element ${m\choose n}\in\ZZ$ induces an element
  ${m\choose n}={m\choose n}1\in k$. The \emph{$n\th$ Hasse
    derivative} $D^n\co k[t^{-1},t]\to k[t^{-1},t]$ is the $k$-linear
  map which satisfies
  \[
    D^n(t^m)={m\choose n}t^{m-n}.
  \]
\end{definition}
Over a field of characteristic $0$,
\[
  D^n = \frac{1}{n!}\frac{d^n}{dt^n}.
\]
\begin{proposition}\label{prop:Hasse-Leibniz}
  The Hasse derivatives satisfy the Leibniz rule
  \[
    D^n\bigl(fg\bigr)=\sum_{i=0}^n \bigl(D^i(f)\bigr)\bigl(D^{n-i}(g)\bigr).
  \]
  Further, for any Laurent polynomial $p(t)$ and any $a\neq 0$, if
  $(D^ip(t))|_{t=a}=0$ for all $i$ then $p(t)=0$.
\end{proposition}
\begin{proof}
  For a proof of the first statement, see, for example,
  Conrad~\cite[Section 4]{Conrad-other-digit}. For the second, suppose
  $(D^ip(t))|_{t=a}=0$ for all $i$. By the first statement, we also
  have $(D^i(t^Np(t)))|_{t=a}=0$ for all $i$, so we may assume that
  $p(t)\in k[t]$. If the highest degree term in $p(t)$ is $b_nt^n$, $b_n\neq
  0$, then $D^n(p(t))|_{t=a}=b_n$.
\end{proof}
See Jeong~\cite{Jeong-other-Hasse} for a recent, more thorough
discussion of what is known about Hasse derivatives, and further
references.

\begin{corollary}\label{cor:Hasse-mx-Leibniz}
  Let $A$ and $B$ be $m\times n$ and $n\times p$ matrices over
  $k[t^{-1},t]$, respectively. Define $D^j(A)$ to be the result of taking the $j\th$
  Hasse derivative of each entry of $A$. Then,
  \[
    D^n(AB)=\sum_{i=0}^n \bigl(D^i(A)\bigr)\bigl(D^{n-i}(B)\bigr).
  \]
\end{corollary}
\begin{proof}
  This is immediate from Proposition~\ref{prop:Hasse-Leibniz} and the
  formula for matrix multiplication.
\end{proof}

\begin{definition}\label{def:t-to-X}
  Let $(C_*,\bdy)$ be a (graded or ungraded) freely generated chain complex
  over $\FF_2[t^{-1},t]$, with a choice of distinguished basis. View $\FF_2$ as an $\FF_2[t^{-1},t]$-algebra via the
  homomorphism sending $t\mapsto 1$ and let
  $C_*^{t=1}=C_*\otimes_{\FF_2[t^{-1},t]}\FF_2$. Define an $\Ainf$-module structure on $C_*^{t=1}$ over $\FF_2[X]/(X^2)$,
  \[
    \mu_{1+n}\co C_*^{t=1}\otimes_{\FF_2} (\FF_2[X]/(X^2))^{\otimes n} \to C_*^{t=1},
  \]
  by declaring that:
  \begin{itemize}[leftmargin=*]
  \item $\mu_1(c)$ is the differential on $C_*$, with $t$ evaluated at
    $1$, and
  \item
    $\mu_{1+n}(c,X,\dots,X)=\bigl(D^n
    \bdy(\overline{c})\bigr)|_{t=1}$. Here, $\overline{c}$ is the
    image of $c$ under the inclusion $C_*^{t=1}\into C_*$ induced by
    the inclusion $\FF_2\into \FF_2[t^{-1},t]$ as the constant polynomials.
  \end{itemize}
  We will say that $(C_*^{t=1},\{\mu_{1+n}\})$ is the \emph{$\Ainf$-module induced by $C_*$.}
\end{definition}

Note that the Hasse derivatives of $\bdy(\overline{c})$ here depend on
the choice of basis for $C_*$ over $\FF_2[t^{-1},t]$, used to
represent $\bdy(\overline{c})$ as a vector or $\bdy$ as a matrix. The
Leibniz rule (Proposition~\ref{prop:Hasse-Leibniz}) also implies a
Leibniz rule for matrix multiplication. In particular, this Hasse
derivative is respected by changing basis by a matrix over $\FF_2$
(i.e., consisting of constant polynomials), though we will not
use this fact directly.

\begin{lemma}\label{lem:t-to-X-behaves}
  Let $(C_*,\bdy)$ be a freely generated chain complex over
  $\FF_2[t^{-1},t]$, with a distinguished basis. Then the
  $\Ainf$-module induced by $C_*$ satisfies the $\Ainf$
  relations. Further, if $f\co C_*\to E_*$ is a quasi-isomorphism of
  freely generated chain complexes over $\FF_2[t^{-1},t]$ then there
  is an induced quasi-isomorphism of $\Ainf$-modules
  $F\co C_*^{t=1}\to E_*^{t=1}$.
\end{lemma}
\begin{proof}
  This follows from Corollary~\ref{cor:Hasse-mx-Leibniz}. For the
  first statement, we need to check that
  \begin{equation}\label{eq:mu-want}
    \sum_{i+j=n} \mu_{1+i}(\mu_{1+j}(c,X,\dots,X),X,\dots,X)=0
  \end{equation}
  for all $n$ and all $c$.
  Consider the $n\th$ Hasse derivative of the matrix
  equation $\bdy^2=0$. By Corollary~\ref{cor:Hasse-mx-Leibniz}, this
  gives
  \[
    \sum_{i+j=n} D^i(\bdy)\circ D^j(\bdy)=0.
  \]
  Setting $t=1$ gives Equation~\eqref{eq:mu-want}. 

  For the second statement, define
  \[
    F_{1+n}\co C_*^{t=1}\otimes_{\FF_2} (\FF_2[X]/(X^2))^{\otimes n} \to E_*^{t=1}
  \]
  by
  \[
    F_{1+n}(c,X,\dots,X)=\bigl(D^n f(\overline{c})\bigr)|_{t=1}.
  \]
  To see that $F$ is an $\Ainf$ homomorphism, we need to check that
  \begin{equation*}
    \sum_{i+j=n} F_{1+i}(\mu_{1+j}(c,X,\dots,X),X,\dots,X)+\mu_{1+i}(F_{1+j}(c,X,\dots,X),X,\dots,X)=0.
  \end{equation*}
  This follows from the equation $f\circ\bdy+\bdy\circ f=0$ by taking
  the $n\th$ Hasse derivative and using
  Corollary~\ref{cor:Hasse-mx-Leibniz}. Now, it follows from the
  universal coefficient theorem (see Lemma~\ref{lem:univ-coeff} above
  for the ungraded case) and the 5-lemma that $F$ is a
  quasi-isomorphism: the map $F_1$ is just the map
  $f\otimes\Id\co C_*\otimes_{\FF_2[t^{-1},t]}\FF_2\to
  E_*\otimes_{\FF_2[t^{-1},t]}\FF_2$ induced by $f$, and we have 
  \[
    \xymatrix{
      0\ar[r] & H(C_*)\otimes_{\FF_2[t^{-1},t]}\FF_2\ar[d]_\cong\ar[r] &
      H(C_*^{t=1})\ar[d]_{(F_1)_*}\ar[r] &
      \Tor^1_{\FF_2[t^{-1},t]}(H(C_*),\FF_2)\ar[d]_\cong \ar[r]& 0\\
      0\ar[r] & H(E_*)\otimes_{\FF_2[t^{-1},t]}\FF_2 \ar[r]&
      H(E_*^{t=1})\ar[r] &
      \Tor^1_{\FF_2[t^{-1},t]}(H(E_*),\FF_2) \ar[r]& 0.
    }
  \]
  This proves the result.
\end{proof}

A priori, the isomorphism type of the $\Ainf$-module $C_*^{t=1}$
depends on the basis for $C_*$ we are working
with. Lemma~\ref{lem:t-to-X-behaves} implies that this dependence is
superficial:
\begin{corollary}\label{cor:basis-indep}
  Up to quasi-isomorphism, the $\Ainf$-module $C_*^{t=1}$ is
  independent of the choice of basis for $C_*$. That is, if $C_*$ is
  isomorphic to $E_*$ then $C_*^{t=1}$ is quasi-isomorphic to $E_*^{t=1}$.
\end{corollary}
(In fact, inspecting the proof a little more shows that the
corollary holds up to isomorphism of $\Ainf$-modules, not just
quasi-isomorphism.)

Let $R=\FF_2[t^{-1},t]$. An element $\zeta\in H_1(Y)/\tors$ induces a homomorphism $\FF_2[H_2(Y)]\to R$, making $R$ into an $\FF_2[H_2(Y)]$-algebra. When thinking of $R$ as an $\FF_2[H_2(Y)]$-algebra, we will denote it $R_\zeta$. When we only want to think of $R$ as a ring, we will drop the subscript $\zeta$.

\begin{proposition}\label{prop:twisted-to-X-Hasse}
  Let $Y$ be a closed $3$-manifold and $\zeta\in H_1(Y)/\tors$. Let
  $\tCFa(Y;R_\zeta)$ be the twisted Floer complex of $Y$ with respect
  to $\zeta$ and $\CFa(Y)$ the untwisted Floer complex, with
  $\Ainf$-module structure over $\FF_2[X]/(X^2)$ induced by
  $\zeta$. Then, there is a quasi-isomorphism of $\Ainf$-modules
  \[
    \CFa(Y)\simeq \tCFa(Y;\FF_2[t^{-1},t]_\zeta)^{t=1}.
  \]
\end{proposition}
\begin{proof}
  This is immediate from Formula~(\ref{eq:t-bdy}) and Lemma~\ref{lem:HF-ainf-formula}.
\end{proof}

\begin{theorem}\label{thm:twist-to-X}
  Let $\tHFa(Y;R_\zeta)$ be the Heegaard Floer
  homology of $Y$ with (twisted) coefficients in $R_\zeta$. As $R$-modules, let
  \[
    \tHFa(Y;R_\zeta)\cong R^{m}\oplus R/(p_1(t))\oplus\cdots\oplus R/(p_n(t))
  \]
  where each $p_i(t)\neq 0$.
  Assume that $p_1(1),\dots,p_k(1)=0$ and
  $p_{k+1}(1),\dots,p_n(1)\neq 0$.
  Then, there is an isomorphism of strictly unital $\Ainf$-modules over $\FF_2[X]/(X^2)$
  \[
    \HFa(Y)\cong \FF_2^m\oplus \FF_2\langle z_1,\dots,z_k,w_1,\dots,w_k\rangle
  \]
  where
  \begin{itemize}[leftmargin=*]
  \item The $\Ainf$-module structure on $\FF_2^m$ and on $\FF_2\langle
    w_1,\dots,w_k\rangle$ is trivial, i.e., for $y\in \FF_2^m\oplus \FF_2\langle
    w_1,\dots,w_k\rangle$ and any $n\geq 0$, 
    \[
      \mu_{1+n}(y,X,\dots,X)=0.
    \]
  \item We have
    \[
      \mu_{1+n}(z_i,X,\dots,X)=\bigl(D^np_i(t)\bigr)|_{t=1}w_i.
    \]
  \end{itemize}
\end{theorem}
\begin{proof}
  First, observe that $\tCFa(Y;R_\zeta)$ decomposes as a direct sum
  of 1-step and 2-step complexes. That is, we can find a basis
  $b_1,\cdots,b_p,c_1,\dots,c_p,d_1,\dots,d_m$ for $\tCFa(Y;R_\zeta)$
  so that $\tbdy(b_i)=q_i(t)c_i$ and $\tbdy(c_i)=\tbdy(d_i)=0$. That
  such a basis exists follows from the proof of
  Proposition~\ref{prop:PID-formal}. Further, we can arrange that
  $q_i(t)=p_i(t)$ for $i\leq n$ and $q_i(t)$ is a unit for $i>n$.

  By Corollary~\ref{cor:basis-indep}, $\tCFa(Y;R_\zeta)^{t=1}$ can be
  computed using this basis, and by
  Proposition~\ref{prop:twisted-to-X-Hasse},
  $\tCFa(Y;R_\zeta)^{t=1}\simeq \CFa(Y)$, as $\Ainf$-modules. So, it
  suffices to consider a single summand $d_i$ or $b_i\stackrel{q_i(t)}{\lra}
  c_i$ of $\tCFa(Y;R_\zeta)^{t=1}$.

  It is immediate from the definitions that the summand generated by
  $d_i$ gives a summand $R$ of $\tHFa(Y;R_\zeta)$ and a summand $\FF_2$
  of $\tCFa(Y;R_\zeta)^{t=1}$.

  A summand of the form $b_i\stackrel{q_i(t)}{\lra} c_i$ gives a copy
  of $R/(q_i(t))$ of $\tHFa(Y;R_\zeta)$ (which is trivial
  if $q_i(t)$ is a unit). From Definition~\ref{def:t-to-X}, a summand
  of the form $b_i\stackrel{q_i(t)}{\lra} c_i$ gives a summand of
  $\tCFa(Y;R_\zeta)^{t=1}$ with trivial homology if $q_i(1)\neq
  0$. If $q_i(1)=0$ then the corresponding summand of
  $\tCFa(Y;R_\zeta)^{t=1}$ is $2$-dimensional, generated by $z_i$ and
  $w_i$, say, has trivial differential, and has 
  \[
    \mu_{1+n}(z_i,X,\dots,X)=\bigl(D^np_i(t)\bigr)|_{t=1}w_i,
  \]
  as claimed.
\end{proof}

\begin{corollary}\label{cor:unrolled-from-twisted}
  With notation as in Theorem~\ref{thm:twist-to-X}, the unrolled
  homology of $\CFa(Y)$ is isomorphic to $\FF_2[Y^{-1},Y]]^{m}$.
\end{corollary}
\begin{proof}
  Since the unrolled homology is invariant under $\Ainf$
  quasi-isomorphism, the unrolled homology of $\CFa(Y)$ is isomorphic
  to the unrolled homology of $\HFa(Y)$. Clearly, the $\FF_2^m\subset
  \HFa(Y)$ survives to give a copy of $\FF_2[Y^{-1},Y]]^{m}$ in the
  unrolled homology. It remains to see that the other summands of
  $\HFa(Y)$ do not contribute to the unrolled homology.

  Consider the summand of $\HFa(Y)$ generated by $z_i$ and $w_i$. By 
  Proposition~\ref{prop:Hasse-Leibniz}, since $p_i(t)\neq 0$, there
  is some integer $k\geq 1$ so that
  $\bigl(D^kp_i(t)\bigr)|_{t=1}\neq 0$. Let $k$ be the first such
  integer. Consider the spectral sequence computing the unrolled
  homology of $\HFa(Y)$, associated to the horizontal
  filtration. Then, on this summand, the first nontrivial differential
  in this spectral sequence is $d_k(z_i)=\alpha Y^kw_i$, where
  $\alpha=\bigl(D^kp_i(t)\bigr)|_{t=1}$. The homology of this
  summand with respect to this differential vanishes.
\end{proof}

\begin{corollary}\label{cor:free-is-nontriv}
  If $\tHFa(Y;R_\zeta)$ has an $\FF_2[t^{-1},t]$-summand then the
  unrolled homology of $\CFa(Y)$ with respect to the action by $\zeta$
  is nontrivial.
\end{corollary}

\begin{corollary}\label{cor:unrolled-is-Novikov}
  The unrolled homology of $\CFa(Y)$ is isomorphic to the completed
  twisted coefficient homology $\tHFa(Y;\FF_2[t^{-1},t]])$.
\end{corollary}
\begin{proof}
  This is immediate from Corollary~\ref{cor:unrolled-from-twisted},
  which computes the unrolled homology in terms of $\tHFa(Y;R_\zeta)$,
  and the universal coefficient theorem, which says that
  \[
    \tHFa(Y;\FF_2[t^{-1},t]])\cong \tHFa(Y;R_\zeta)\otimes_R
    \FF_2[t^{-1},t]].
  \]
  (Recall that $\FF_2[t^{-1},t]]$ is flat over $\FF_2[t^{-1},t]$.)
\end{proof}

While we will not need it for our application, we conclude this
section by noting a more homological-algebraic interpretation of
Proposition~\ref{prop:twisted-to-X-Hasse}. View $\FF_2$ as an
$R=\FF_2[t^{-1},t]$-module in the usual way, by letting $t$ act by
$1$. Then, $\FF_2$ has a 2-step free resolution over $R$:
\[
  R\stackrel{1-t}{\lra} R.
\]
From this, it is straightforward to compute that
$\Ext_R(\FF_2,\FF_2)\cong \FF_2[X]/(X^2)$ (and, in fact, $\RHom_R(\FF_2,\FF_2)$
is quasi-isomorphic to $\FF_2[X]/(X^2)$).

For any chain complex $C_*$ over $R$, there is an $\Ainf$ action of
$\Ext_R(\FF_2,\FF_2)$ on $\Tor_R(C_*,\FF_2)$. Explicitly, $\Tor_R(C_*,\FF_2)$ is
the homology of the total complex of the bicomplex
\begin{equation}\label{eq:1t-bicomplex}
  0\lra C\stackrel{1-t}{\lra} C\lra 0.
\end{equation}
The element $X$ shifts this bicomplex one unit to the right, i.e.,
sends the first copy of $C$ to the second by the identity map and
sends the second copy of $C$ to $0$. So, this total complex is a
differential module over $\FF_2[X]/(X^2)$, and its homology
$\Tor_R(C_*,\FF_2)$ inherits the structure of an $\Ainf$-module over
$\FF_2[X]/(X^2)$.

\begin{theorem}\label{thm:Koszul}
  For any finitely generated, free chain complex $C_*$ over $R$, there
  is a quasi-isomorphism of $\Ainf$-modules over $\FF_2[X]/(X^2)$
  \[
    C_*^{t=1}\simeq \Tor_R(C_*,\FF_2).
  \]
  In particular, as $\Ainf$-modules over $\FF_2[X]/(X^2)$, 
  \[
    \HFa(Y)\simeq \Tor_{\FF_2[t^{-1},t]}(\tHFa(Y),\FF_2).
  \]
\end{theorem}
\begin{proof}
  As in the proof of Theorem~\ref{thm:twist-to-X}, it suffices to
  prove the result when $C_*$ consists of a single generator $d$ or a
  pair of generators $b,c$ with differential
  $b\stackrel{p(t)}{\lra} c$. In the first case, it is straightforward
  to see that both $C_*^{t=1}$ and $\Tor_R(C_*,\FF_2)$ are isomorphic to
  $\FF_2$ with trivial $\Ainf$-module structure. In the second case, if
  $p(1)\neq 0$ then $C_*^{t=1}$ is acyclic and $\Tor_R(C_*,R)=0$. So,
  it remains to verify the second case under the assumption that
  $p(1)=0$.

  Let $E_*$ be the total complex of the
  bicomplex~\eqref{eq:1t-bicomplex}.  We will construct an $\Ainf$
  quasi-isomorphism $f\co C_*^{t=1}\to E_*$.

  To fix notation, write $E_*=C\otimes \FF_2[X]/(X^2)$, with
  differential
  \[
    \begin{bmatrix}
      \bdy_C & (1-t)\\
      0 & \bdy_C
    \end{bmatrix}.
  \]
  That is, the complex $E_*$ is the total complex of the square
  \[
    \mathcenter{
    \xymatrix{
      R\ar[r]^-{1-t}\ar[d]_{p(t)} & R\ar[d]^{p(t)}\\
      R\ar[r]_-{1-t} & R.
    }}
  \ =\ 
  \mathcenter{
    \xymatrix{
    R\langle b\rangle\ar[r]^-{1-t}\ar[d]_{p(t)} & R\langle Xb\rangle\ar[d]^{p(t)}\\
      R\langle c\rangle\ar[r]_-{1-t} & R\langle Xc\rangle.
    }}
  \]
  
  Define Laurent polynomials $q_n(t)$ inductively by
  \begin{align*}
    q_1(t)&=p(t)/(1-t)\\
    q_{n+1}(t)&=(q_n(1)-q_n(t))/(1-t).
  \end{align*}
  The fact that $q_1(t)$ is a Laurent polynomial follows from the
  restriction that $p(1)=0$.

  We claim that 
  \[
    q_n(1)=\bigl(D^np(t)\bigr)|_{t=1}.
  \]
  By induction, we have
  \[
    p(t)=(t-1)q_1(1)+(t-1)^2q_2(1)+\cdots+(t-1)^{n-1}q_{n-1}(1)+(t-1)^nq_n(t)
  \]
  (cf.\ Taylor's theorem).  Hence,
  \[
    D^np(t)=(D^n(t-1)^n)q_n(t)+(t-1)r(t)=q_n(t)+(t-1)r(t).
  \]
  Evaluating at $1$ verifies the claim.
  
  Now, define:
  \begin{align*}
    f_1(c)&=Xc\\
    f_{1+n}(c,X,\dots,X)&=0 & n>0\\
    f_1(b)&=Xb + q_1(t)c\\
    f_{1+n}(b,X,\dots,X)&=q_{n+1}(t)c & n>0.
  \end{align*}
  It is straightforward to see that $f_1$ is a quasi-isomorphism. We
  claim that the $f_i$ satisfy the $\Ainf$ homomorphism relations;
  this finishes the proof. We must check that for $y\in\{b,c\}$,
  \[
    \sum_{i+j=n}\mu_{1+i}^E(f_{1+j}(y,X,\dots,X),X,\dots,X)+f_{1+i}(\mu_{1+j}^{C^{t=0}}(y,X,\dots,X),X,\dots,X)=0.
  \]
  Recall that $\mu_{1+i}^E=0$ for $i>1$. In the case $y=c$, each
  term in the equation vanishes. For $y=b$ and $n=0$, the
  left side of the equation is
  \[
    \bdy^E(f_1(b))+f_1(\bdy^{C^{t=0}}(b))=\bdy^E(Xb+q_1(t)c)=(p(t)+(1-t)q_1(t))Xc=0.
  \]
  For $y=b$ and $n>0$, the left side is equal to
  \begin{align*}
    &\bdy^E(f_{1+n}(b,X,\dots,X)+\mu_2^E(f_{n}(b,X,\dots,X),X)+\sum_{i+j=n}f_{1+i}((D^jp(t))|_{t=1}c,X,\dots,X)\\
    &=\bdy^E(q_{n+1}(t)c)+\mu_2^E(q_n(t)c,X)+(D^np(t))|_{t=1}Xc\\
    &=q_{n+1}(t)(1-t)Xc+q_n(t)Xc+q_n(1)Xc\\
    &=(q_n(1)-q_n(t)+q_n(t)+q_n(1))Xc\\
    &=0,
  \end{align*}
  as desired.
\end{proof}

\begin{remark}
  Presumably, one can give a direct proof of
  Theorem~\ref{thm:Koszul}, without relying on the classification of
  finitely generated modules over a PID, but the computations required
  seem involved.
\end{remark}

\begin{remark}
  For simplicity, we have worked in characteristic $2$ and focused on
  the action of a single element $\zeta\in H_1(Y)/\tors$, but we
  expect that the results in this section generalize to the entire
  $\Ainf$-module structure over $\Lambda^* H_1(Y)/\tors$ over $\ZZ$ (though some
  of the proofs do not).
\end{remark}

\section{The module structure on Khovanov homology and the Ozsv\'ath-Szab\'o spectral sequence}\label{sec:Kh-module}
\subsection{Definition and invariance of the basepoint action on
  Khovanov homology}
Fix a link diagram $L$ and a basepoint $q\in L$ not at any of the
crossings. (From here on, \emph{basepoint} means ``basepoint not at a
crossing.'') The Khovanov complex $\KhCx(L)$ of $L$ inherits the
structure of a module over $\FF_2[X]/(X^2)$ as follows. A generator of
$\KhCx(L)$ is a complete resolution of $L$ and a decoration of each
component of the resolution by $1$ or $X$. Multiplication by $X$ on a
generator of $\KhCx(L)$:
\begin{itemize}[leftmargin=*]
\item is zero if the generator labels the circle containing $q$ by $X$ and
\item changes the label on the circle containing $q$ to $X$, if the generator
  labels the circle containing $q$ by $1$.
\end{itemize}
It is straightforward to check that multiplication by $X$ is a chain
map.  The action of $X$ preserves the homological grading
and decreases the quantum grading by $2$.

Given two basepoints $p,q\in L$, the actions at $p$ and $q$ commute,
and hence make
$\KhCx(L)$ into a differential bimodule over $\FF_2[W]/(W^2)$ and
$\FF_2[X]/(X^2)$ or, equivalently, a differential module over
$\FF_2[W,X]/(W^2,X^2)$. Note that while $\KhCx(L)$ is free over
$\FF_2[X]/(X^2)$, it is typically not free over
$\FF_2[W,X]/(W^2,X^2)$.

Let $\Sigma^{a,b}$ denote shifting the homological grading up by $a$
and the quantum grading up by $b$.

Given a basepoint $p$ on $L$, the \emph{reduced Khovanov complex}
$\rKhCx(L)$ is the subcomplex of $\Sigma^{0,1}\KhCx(L)$ where the
circle containing $p$ is labeled $X$ or, equivalently, the quotient
complex of $\Sigma^{0,-1}\KhCx(L)$ where the circle containing $p$ is
labeled $1$. Given a second basepoint $q$ on $L$, $\rKhCx(L)$ inherits
a module structure over $\FF_2[X]/(X^2)$.

We will use the following lemma, to avoid writing the
same proof twice in Theorem~\ref{thm:Kh-mod-well-defd}.
\begin{lemma}\label{lem:reduce-Kh}
  Let $L$ be a link and $p,q$ basepoints on $L$. Write $\KhCx(L)_{p,q}$
  for the Khovanov complex $\KhCx(L)$ viewed as a bimodule over
  $\FF_2[W]/(W^2)$ and $\FF_2[X]/(X^2)$ via the basepoints $p$ and $q$. Write
  $\rKhCx(L)_{p,q}$ for the reduced Khovanov complex,
  reduced at $p$ and viewed as a module over $\FF_2[X]/(X^2)$ via the
  basepoint $q$. View $\FF_2$ as a module over
  $\FF_2[W]/(W^2)$ where $W$ acts by $0$. Then, there is a chain
  isomorphism of $\FF_2[X]/(X^2)$-modules
  \[
    \rKhCx(L)_{p,q}\cong \Sigma^{0,-1}\KhCx(L)_{p,q}\otimes_{\FF_2[W]/(W^2)}\FF_2.
  \]
  Further, $\KhCx(L)$ is a free module over $\FF_2[W]/(W^2)$, so $\rKhCx(L)$ is
  quasi-isomorphic to the $\Ainf$ tensor product of the $\Ainf$-bimodule
  $\Kh(L)$ and the $\FF_2[W]/(W^2)$ module $\FF_2$.
\end{lemma}
\begin{proof}
  The first statement is immediate from the definitions. The second
  follows from the facts that the $\Ainf$ tensor product is invariant
  under $\Ainf$ homotopy equivalence and that any $\Ainf$-(bi)module is
  $\Ainf$ homotopy equivalent to its homology (Proposition~\ref{prop:hpt}).
\end{proof}

\begin{theorem}\label{thm:Kh-mod-well-defd}
  Let $L$ be a link and $p,p',q,q'$ points on $L$ so that $p,p'$ lie
  on the same component of $L$ and $q,q'$ lie on the same component of
  $L$.
  \begin{enumerate}[leftmargin=*]
  \item Write $\KhCx(L)_{p,q}$ (respectively $\KhCx(L)_{p',q'}$) for
    the Khovanov complex $\KhCx(L)$ viewed as a module over
    $\FF_2[W,X]/(W^2,X^2)$ via the basepoints $p,q$ (respectively
    $p',q'$). Then, $\KhCx(L)_{p,q}$ is quasi-isomorphic to
    $\KhCx(L)_{p',q'}$.
  \item Write $\rKhCx(L)_{p,q}$ (respectively $\KhCx(L)_{p',q'}$) for
    the reduced Khovanov complex $\KhCx(L)$, reduced at $p$
    (respectively $p'$), and viewed as a module over $\FF_2[X]/(X^2)$
    via the basepoint $q$ (respectively $q'$). Then, $\rKhCx(L)_{p,q}$
    is quasi-isomorphic to $\rKhCx(L)_{p',q'}$.
  \end{enumerate}
\end{theorem}

The fact that the Khovanov homology has a well-defined
structure of a bimodule appears in Hedden-Ni~\cite[Proposition
1]{HN-kh-detects}, via a similar argument. The fact that the
action of $\ZZ[X]/(X^2)$ at a single basepoint is well-defined is due
to Khovanov~\cite[Section 3]{Kho-kh-patterns}, by a different argument
which apparently does not generalize.

\begin{proof}
  For the first half of the theorem, by Proposition~\ref{prop:qi-is-he}, it
  suffices to show that moving $q$ past a crossing $C$ gives an $\Ainf$
  quasi-isomorphic bimodule over $\FF_2[W]/(W^2)$ and
  $\FF_2[X]/(X^2)$. An $\Ainf$-bimodule map from $M$ to $N$ consists of maps
  \[
    f_{m,1,n}\co (\FF_2[W]/(W^2))^{\otimes m}\otimes M\otimes
    (\FF_2[X]/(X^2))^{\otimes n}\to N.
  \]
  Our quasi-isomorphism $f$ will have $f_{0,1,0}=\Id$ and
  $f_{m,1,n}=0$ if $m>0$ or $n>1$. To define $f_{0,1,1}$, we need some
  more notation.

  Write a generator of the Khovanov complex as a pair $(v,x)$ where
  $v\in\{0,1\}^c$ (where $c$ is the number of crossings of $L$) and $x$ is a labeling of the circles of the $v$-resolution $L_v$ by $1$
  or $X$. That is, $x\in \{1,X\}^{\pi_0(L_v)}$. Let $|v|=\sum v$ be
  the height of the vertex $v$, and let $\leq$ denote the partial
  order on the cube $\{0,1\}^c$. Given $v,w$ with $|v|-|w|=\pm 1$ and
  $v<w$ or $w<v$;
  $x\in \{1,X\}^{\pi_0(L_v)}$; and $y\in \{1,X\}^{\pi_0(L_w)}$, let
  $n_{x,y}$ be $1$ if the following conditions are satisfied:
  \begin{itemize}[leftmargin=*]
  \item For $Z\in L_v\cap L_w$, $x(Z)=y(Z)$;
  \item If there are two circles $Z_1,Z_2$ in $L_v$ which are merged
    into a circle $Z$ in $L_w$ then $y(Z)=x(Z_1)x(Z_2)$ (where the
    multiplication is in $\FF_2[X]/(X^2)$); and
  \item If there is one circle $Z$ in $L_v$ which is split into two
    circles $Z_1,Z_2$ in $L_w$ then $y(Z_1)\otimes
    y(Z_2)=\Delta(x(Z))$ where $\Delta\co \FF_2[X]/(X^2)\to
    \FF_2[X]/(X^2)\otimes \FF_2[X]/(X^2)$ is the comultiplication
    \[
      \Delta(1)=1\otimes X+X\otimes 1\qquad\qquad \Delta(X)=X\otimes X.
    \]
  \end{itemize}
  Define $n_{x,y}=0$ otherwise. Then, the Khovanov differential is
  \[
    \delta(x,v)=\sum_{(y,w)\mid w\geq v,\ |v|+1=|w|} n_{x,y} \cdot (y,w).
  \]

  Define the map $f_{0,1,1}$ by
  \[
    f_{0,1,1}((y,w),X)=\sum_{(x,v)\mid w\geq v,\
      |v|+1=|w|,\ v(C)\neq w(C)}n_{y,x}\cdot (x,v).
  \]
  Equivalently, $f_{0,1,1}$ comes from performing the differential at
  $C$ backwards. That is, if $\overline{C}$ is the opposite crossing
  to $C$ then $f_{0,1,1}$ is the part of the differential associated
  to changing $\overline{C}$. (Compare~\cite[Lemma
  2.3]{HN-kh-detects},~\cite[Section 2.2]{BS-kh-splitting}.) See
  Figures~\ref{fig:bp-move-map} and~\ref{fig:bp-move-map-2}.

  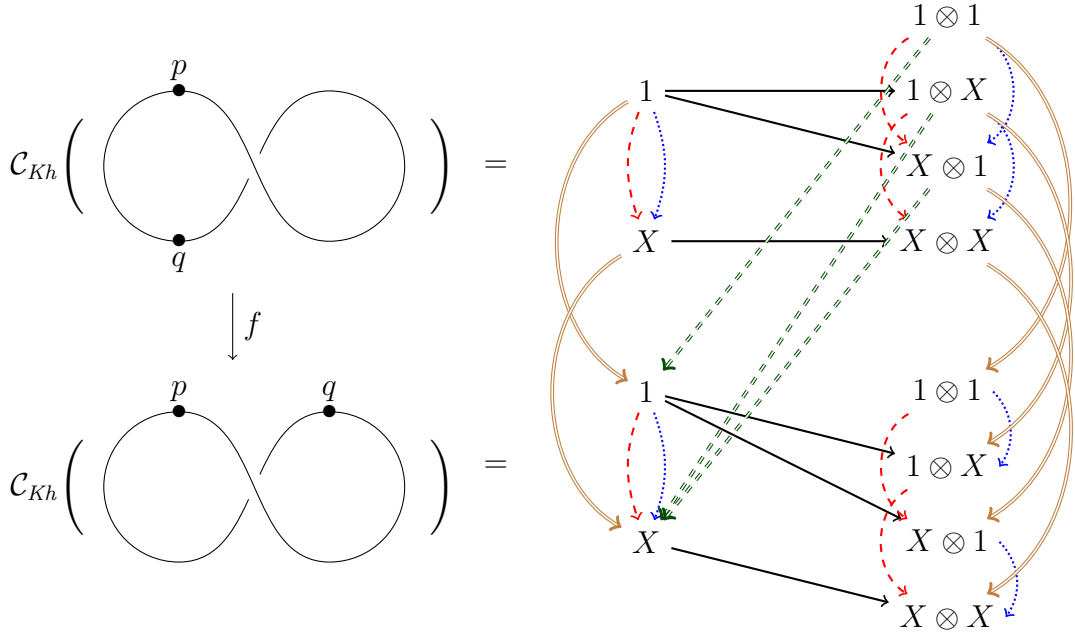
\begin{figure}
    \centering
    \begin{tikzpicture}

      \foreach \name/\sname/\shift/\sshift in {topunknot/t/0/0,bottomunknot/b/-4/-1}{
        \begin{scope}[yshift=\shift cm]

          \node at (-3.5,0) (\name) {
            \begin{tikzpicture}
              \node at (-2.75,0) (CKh1) {$\KhCx\Bigg($};
              \draw (1,-1) to[out=180,in=0] (-1,1);
              \fill[white] (0,0) circle (1ex);
              \draw (-1,1) arc(90:270:1) (-1,-1) to[out=0,in=180] (1,1) arc(90:-90:1);
              \node at (2.5,0) (rparen1) {$\Bigg)$};

              \node at (-1,1) {$\bullet$};  
              \node[anchor=south] at (-1,1) {$p$};

              \if\sname t
              \node at (-1,-1) {$\bullet$};
              \node[anchor=north] at (-1,-1) {$q$};
              \else
              \node at (1,1) {$\bullet$};
              \node[anchor=south] at (1,1) {$q$};
              \fi

            \end{tikzpicture}
          };
          
          \node at (0,0) (equals1) {$=$};

          \node at (6,1) (\sname 1) {$1$}; 
          \node at (6,-1) (\sname X) {$X$}; 

          \begin{scope}[yshift=\sshift cm]
            \node at (2,2) (\sname 11) {$1\otimes 1$};
            \node at (2,1) (\sname 1X) {$1\otimes X$};
            \node at (2,0) (\sname X1) {$X\otimes 1$};
            \node at (2,-1) (\sname XX) {$X\otimes X$};
          \end{scope}

          \draw[diffa] (\sname 11) to (\sname 1);
          \draw[diffa] (\sname 1X) to (\sname X);
          \draw[diffa] (\sname X1) to (\sname X);

          \draw[Wa, bend right=20] (\sname 1) to (\sname X);
          \draw[Wa, bend right=60] (\sname 11) to (\sname X1);
          \draw[Wa, bend right=60] (\sname 1X) to (\sname XX);

          \draw[Xa, bend left=20] (\sname 1) to (\sname X);

          \if\sname t
          \draw[Xa, bend left=60] (t11) to (tX1);
          \draw[Xa, bend left=60] (t1X) to (tXX);
          \else
          \draw[Xa, bend left=60] (b11.east) to (b1X.east);
          \draw[Xa, bend left=60] (bX1.east) to (bXX.east);
          \fi

        \end{scope}}

      \draw[->] (topunknot) to node[right]{$f$} (bottomunknot);

      \draw[f1a, bend right=60] (t11) to (b11);
      \draw[f1a, bend right=60] (t1X) to (b1X);
      \draw[f1a, bend right=60] (tX1) to (bX1);
      \draw[f1a, bend right=60] (tXX) to (bXX);
      \draw[f1a, bend left=60] (t1) to (b1);
      \draw[f1a, bend left=60] (tX) to (bX);
      \draw[f2a] (t1) to (b1X);
      \draw[f2a] (t1) to (bX1);
      \draw[f2a] (tX) to (bXX);
    \end{tikzpicture}
    \caption{\textbf{Example of the map $f$.} The solid arrows are the
      differential, the {dashed arrows} are the action
      of $W$, the {dotted arrows} are the action of
      $X$, the {double arrows} are the map $f_{0,1,0}$,
      and the {dashed double arrows} are the map
      $f_{0,1,1}(\cdot,X)$. The $\Ainf$ relations correspond to certain ways of
      getting from one vertex to another in two steps.}
    \label{fig:bp-move-map}
  \end{figure}

  \begin{figure}
    \centering
    \begin{tikzpicture}

      \foreach \name/\sname/\shift/\sshift in {topunknot/t/0/0,bottomunknot/b/-4/-1}{
        \begin{scope}[yshift=\shift cm]

          \node (\name) at (-3.5,0) (\name) {
            \begin{tikzpicture}
              \node at (-2.75,0) {$\KhCx\Bigg($};
              \draw (-1,1) arc(90:270:1) (-1,-1) to[out=0,in=180] (1,1) arc(90:-90:1);
              \fill[white] (0,0) circle (1ex);
              \draw (1,-1) to[out=180,in=0] (-1,1);
              \node at (2.5,0) (rparen1) {$\Bigg)$};

              \node at (-1,1) {$\bullet$};  
              \node[anchor=south] at (-1,1) {$p$};

              \if\sname t
              \node at (-1,-1) {$\bullet$};
              \node[anchor=north] at (-1,-1) {$q$};
              \else
              \node at (1,1) {$\bullet$};
              \node[anchor=south] at (1,1) {$q$};
              \fi

            \end{tikzpicture}
          };

          \node at (0,0) (equals1) {$=$};
          
          \node at (2,1) (\sname 1) {$1$}; 
          \node at (2,-1) (\sname X) {$X$}; 

          \begin{scope}[yshift=\sshift cm]
            \node at (6,2) (\sname 11) {$1\otimes 1$};
            \node at (6,1) (\sname 1X) {$1\otimes X$};
            \node at (6,0) (\sname X1) {$X\otimes 1$};
            \node at (6,-1) (\sname XX) {$X\otimes X$};
          \end{scope}

          \draw[diffa] (\sname 1) to (\sname 1X);
          \draw[diffa] (\sname 1) to (\sname X1);
          \draw[diffa] (\sname X) to (\sname XX);

          \draw[Wa, bend right=20] (\sname 1) to (\sname X);
          \draw[Wa, bend right=60] (\sname 11) to (\sname X1);
          \draw[Wa, bend right=60] (\sname 1X) to (\sname XX);

          \draw[Xa, bend left=20] (\sname 1) to (\sname X);
          \if\sname t
          \draw[Xa, bend left=60] (t11) to (tX1);
          \draw[Xa, bend left=60] (t1X) to (tXX);
          \else
          \draw[Xa, bend left=40] (b11.east) to (b1X.east);
          \draw[Xa, bend left=40] (bX1.east) to (bXX.east);
          \fi

        \end{scope}}

      \draw[->] (topunknot) to node[right]{$f$} (bottomunknot);

      \draw[f1a, bend left=60] (t11) to (b11);
      \draw[f1a, bend left=60] (t1X) to (b1X);
      \draw[f1a, bend left=60] (tX1) to (bX1);
      \draw[f1a, bend left=60] (tXX) to (bXX);
      \draw[f1a, bend right=60] (t1) to (b1);
      \draw[f1a, bend right=60] (tX) to (bX);
      \draw[f2a] (t11) to (b1);
      \draw[f2a] (t1X) to (bX);
      \draw[f2a] (tX1) to (bX);
    \end{tikzpicture}
    \caption{\textbf{A second example of the map $f$.} Notation is
    the same as in Figure~\ref{fig:bp-move-map}.}
    \label{fig:bp-move-map-2}
  \end{figure}
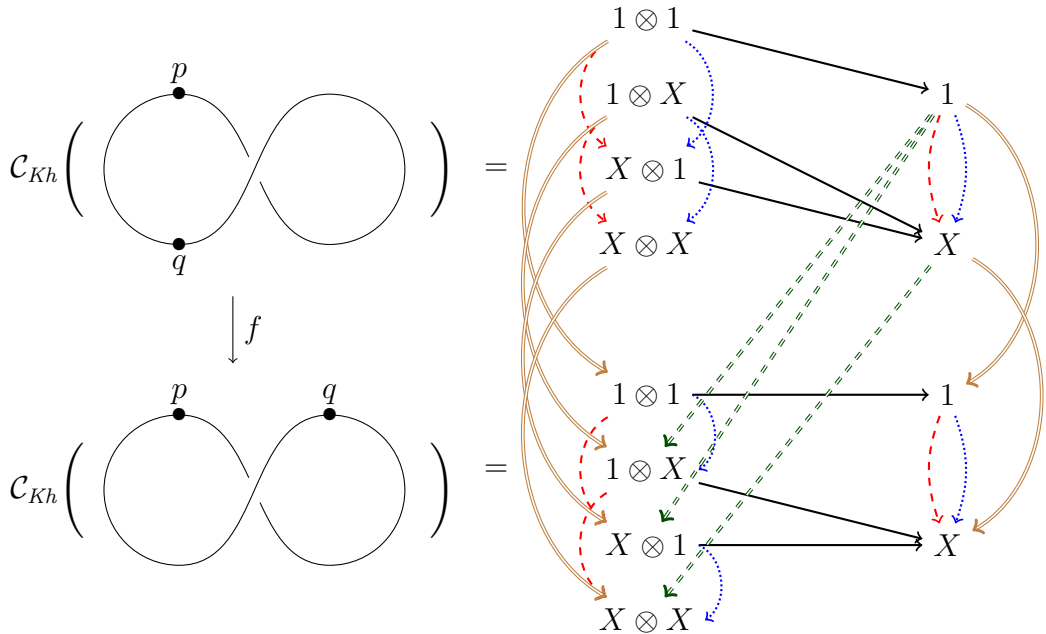
  
  The nontrivial $\Ainf$ relations to verify are the following:
  \begin{align}
    \mu_{0,1,0}(f_{0,1,0}((v,x)))+f_{0,1,0}(\mu_{0,1,0}((v,x)))&=0\label{eq:Kh-Ainf-1}\\
      \mu_{1,1,0}(W,f_{0,1,0}((v,x)))+f_{0,1,0}(\mu_{1,1,0}(W,(v,x)))&=0\label{eq:Kh-Ainf-2}\\
    \begin{split}
      \mu_{0,1,1}(f_{0,1,0}((v,x)),X)+f_{0,1,0}(\mu_{0,1,1}((v,x),X))\qquad&\\+\mu_{0,1,0}(f_{0,1,1}((v,x),X))+f_{0,1,1}(\mu_{0,1,0}((v,x)),X)&=0
    \end{split}
\label{eq:Kh-Ainf-3}\\
    \mu_{1,1,0}(W,f_{0,1,1}((v,x),X))+f_{0,1,1}(\mu_{1,1,0}(W,(v,x)),X)&=0\label{eq:Kh-Ainf-5}\\
    \mu_{0,1,1}(f_{0,1,1}((v,x),X),X)+f_{0,1,1}(\mu_{0,1,1}((v,x),X),X)&=0\label{eq:Kh-Ainf-4}
  \end{align}
  All other $\Ainf$ relations are automatically satisfied because all
  of the terms vanish. In these equations, we have dropped terms
  which automatically vanish, like
  $\mu_{0,1,0}(f_{1,1,0}(W,(v,x)))$, to keep the expressions
  shorter.

  Since $f_{0,1,0}$ is the identity map,
  Equations~\eqref{eq:Kh-Ainf-1} and~\eqref{eq:Kh-Ainf-2} are
  obviously satisfied. Equation~\eqref{eq:Kh-Ainf-3} was checked in
  Hedden-Ni's paper
  \cite[Equation~(1)]{HN-kh-detects}.  For
  Equation~\eqref{eq:Kh-Ainf-5}, let $L'$ be the result of replacing
  the crossing $C$ in $L$ by the opposite crossing $C'$; then, the
  equation follows from the fact that the differential on the Khovanov
  complex of $L'$ respects the $\FF_2[W]/(W^2)$-module structure. 

  For Equation~\eqref{eq:Kh-Ainf-4}, consider the coefficient of
  $(w,y)$. For the coefficient to be non-zero, $v$ must be
  obtained from $w$ by changing the entry corresponding to $C$ from
  $0$ to $1$. Suppose first that two circles $Z_1,Z_2$ in $L_w$ merge
  into one circle $Z$ in $L_v$.  Note that one circle, say $Z_1$, must
  contain $q$ and the other circle $Z_2$ must contain $q'$. Then,
  recording only the labels of the circles $Z_1$, $Z_2$, and $Z$,
  $\mu_{0,1,1}(f_{0,1,1}((v,x),X),X)$ is given by
  \begin{align*}
    1 &\mapsto 1\otimes X+X\otimes 1 \mapsto X\otimes X\\
    X &\mapsto X\otimes X\mapsto 0,
  \end{align*}
  while $f_{0,1,1}(\mu_{0,1,1}((v,x),X),X)$ is given by
  \begin{align*}
    1 &\mapsto X \mapsto X\otimes X\\
    X &\mapsto 0 \mapsto 0.
  \end{align*}
  So, these two terms cancel. Next, suppose that one circle $Z$ in
  $L_w$ splits into two circles $Z_1,Z_2$ in $L_v$, where $Z_1$
  contains $q$ and $Z_2$ contains $q'$. Then,
  $\mu_{0,1,1}(f_{0,1,1}((v,x),X),X)$ is given by
  \begin{align*}
    1\otimes 1 & \mapsto 1\mapsto X &
    1\otimes X & \mapsto X\mapsto 0 \\
    X\otimes 1 & \mapsto X\mapsto 0 &
    X\otimes X & \mapsto 0\mapsto 0\\
  \shortintertext{while  $f_{0,1,1}(\mu_{0,1,1}((v,x),X),X)$ is given by}
    1\otimes 1 & \mapsto X\otimes 1\mapsto X &
    1\otimes X & \mapsto X\otimes X\mapsto 0\\
    X\otimes 1 & \mapsto 0\mapsto 0 &
    X\otimes X & \mapsto 0\mapsto 0.
  \end{align*}
  So, again, these two terms cancel. These two cases are illustrated
  in Figures~\ref{fig:bp-move-map} and \ref{fig:bp-move-map-2}. (An
  alternative way to view this identity is as follows: The map
  $f_{0,1,1}\from \KhCx(L_v)\to\KhCx(L_w)$ is induced by an elementary
  saddle cobordism, while the $X$-actions on $\KhCx(L_v)$ and
  $\KhCx(L_w)$ are induced by identity cobordisms decorated with a
  single `dot'; then, either term of Equation~\eqref{eq:Kh-Ainf-4}
  corresponds to an elementary saddle cobordism decorated with a
  single dot on the saddle component.)
  
  This proves the first half of the theorem.
  The second half of the theorem follows from the first and
  Lemma~\ref{lem:reduce-Kh} (and the fact that the $\Ainf$ tensor
  product is invariant under $\Ainf$ quasi-isomorphisms).
\end{proof}

\begin{corollary}\label{cor:Kh-invariance}
  Let $L$ be a link and $p,p',q,q'$ points on $L$ so that $p,p'$ lie
  on the same component of $L$ and $q,q'$ lie on the same component of
  $L$.  Up to $\Ainf$-isomorphism, the $\Ainf$-modules $\Kh(L)_p$ and
  $\Kh(L)_{p'}$ (respectively $\rKh(L)_{p,q}$ and $\rKh(L)_{p',q'}$)
  over $\FF_2[X]/(X^2)$ are isomorphic. In fact, the isomorphism
  classes of these $\Ainf$-modules are invariants of the isotopy
  classes of the triple $(L,p)$ and $(L,p,q)$, respectively.
\end{corollary}
\begin{proof}
  The first statement is immediate from
  Theorem~\ref{thm:Kh-mod-well-defd} and homological perturbation
  theory. For the second statement, it suffices to verify invariance
  under Reidemeister moves disjoint from the basepoints and moving a
  strand across a basepoint. Invariance under Reidemeister moves
  disjoint from the basepoints was proved by
  Khovanov~\cite{Kho-kh-categorification}, and invariance under moving
  a strand across a basepoint is a special case of the first half of
  the corollary.
\end{proof}

\subsection{The Ozsv\'ath-Szab\'o spectral sequence respects the \texorpdfstring{$\Ainf$-}{A-infinity }module structure}
\label{sec:ss-respect}
Let $L$ be a link in $S^3$ and $p,q\in L$. Choose an arc
$\gamma\subset S^3\setminus L$ from $p$ to $q$. The preimage
$\zeta\subset \Sigma(L)$ of $\gamma$ is a simple closed curve,
representing an element of $H_1(\Sigma(L))$. The homology class
represented by $\zeta$ is independent of the choice of $\gamma$ since
isotoping $\gamma$ across $L$ changes $\zeta$ by the preimage of a
meridian of $L$, which bounds a disk in $\Sigma(L)$.

The homology class $[\zeta]$ makes $\CFa(\Sigma(L);\FF_2)$ into a
module over $\FF_2[X]/(X^2)$, as described in
Section~\ref{sec:HF-module}. (Of course, if
$[\zeta]\in H_1(\Sigma(L))$ is torsion---for example, if $p$ and $q$
lie on the same component of $L$ or if $\Sigma(L)$ is a rational
homology sphere---then this module structure is trivial.)

In the following proposition, by a \emph{filtration} we mean a
descending filtration, i.e., a sequence of submodules
$C=F^0\supset F^1\supset F^2\supset\cdots$.
\begin{proposition}\label{prop:ss-resp-X}
  Let $L$ be a link in $S^3$ and $p,q$ points on $L$. There is an
  ungraded, filtered $\Ainf$-module $(C,\{\mu_n\})$ over
  $\FF_2[X]/(X^2)$ with the following properties:
  \begin{enumerate}[label=(\arabic*),leftmargin=*]
  \item Forgetting the filtration, $C$ is quasi-isomorphic to
    $\CFa(\Sigma(L))$ as an $\Ainf$-module over $\FF_2[X]/(X^2)$.
  \item The differential $\mu_1$ on $C$ strictly increases the filtration.
  \item There is an isomorphism of $\FF_2$-modules
    $g\co C\stackrel{\cong}{\lra} \rKhCx(m(L))$, taking the filtration
    on $C$ to the homological grading on $\rKhCx(m(L))$.
  \item\label{item:mu1-first-order} To first order, $\mu_1$ agrees with the Khovanov
    differential. That is,
    \[
      \mu_1-g^{-1}\circ \bdy_{\KhCx}\circ g
    \]
    increases the filtration by at least $2$.
  \item\label{item:mu2-zeroth-order} To zeroth order, the operation
    $\mu_2(\cdot,X)$ on $C$ agrees with the action of $X$ on the
    Khovanov complex. That is,
    \[
      y\mapsto \mu_2(y,X) - g^{-1}(g(y)\cdot X)
    \]
    increases the filtration by at least $1$.  (Note that the Khovanov
    multiplication actually respects the homological grading.)
  \end{enumerate}
\end{proposition}
\begin{proof}
  This is essentially Hedden-Ni's refinement~\cite[Theorem
  4.5]{HN-kh-detects} of Ozsv\'ath-Szab\'o's construction of the
  spectral sequence for a branched double cover~\cite[Theorem
  6.3]{OSz-hf-dcov}.  The only additional assertions are that there is
  an $\Ainf$-module structure on $C$ and the quasi-isomorphism
  between $C$ and $\CFa(\Sigma(L))$ extends to an $\Ainf$
  homomorphism. So, we will only explain the additional steps required
  to adapt Hedden-Ni's proof, and will adopt much of their notation
  without re-introducing it.

  Throughout this proof, Floer complexes are with $\FF_2$-coefficients
  (which we suppress from the notation).
  
  Let $c$ be the number of crossings of $L$ and, given
  $I\in \{0,1,\infty\}^c$, let
  \[
    \HD^I=(\Sigma,\alphas,\betas^I,z)
  \]
  be the Heegaard diagram considered by Ozsv\'ath-Szab\'o and
  Hedden-Ni. Fix
  \begin{itemize}[leftmargin=*]
  \item a curve $\zeta$ in $\Sigma$ representing the homology class of a lift of an arc from $p$ to $q$,
  \item small pushoffs $A_2,A_3,\dots$ of $A_1=\zeta\cap\alphas$ as in
    Section~\ref{sec:holomorphic}, and
  \item a collection of sufficiently generic almost complex structures.
  \end{itemize}
  
  Given a sequence $I_0<I_1<\cdots<I_m$ of immediate successors in
  $\{0,1,\infty\}^c$ and an integer $n\geq 0$, define a map
  \[
    \mu_{1+n}^{I_0<\cdots<I_m}(\cdot,\overbrace{X,\cdots,X}^n)\co \CFa(\HD^{I_0})\to \CFa(\HD^{I_m})
  \]
  by counting rigid holomorphic $(m+2)$-gons in the Heegaard multi-diagram
  \[
    (\Sigma,\alphas,\betas^{I_0},\dots,\betas^{I_m},z)
  \]
  with point constraints along the $\alpha$-boundary coming from
  $A_1,\dots,A_n$ and corners at some generator
  $x\in T_\alpha\cap T_{\beta^{I_0}}$, some generator
  $y\in T_\alpha\cap T_{\beta^{I_m}}$, and the top generators
  $\Theta_1,\dots,\Theta_m$ for
  $(\betas^{I_0},\betas^{I_1}),\dots,(\betas^{I_{m-1}},\betas^{I_m})$.

  Let
  \[
    \mu_{1+n}\co \bigoplus_{I\in\{0,1,\infty\}^c}\CFa(\HD^I)\otimes\bigl(\FF_2[X]/(X^2)\bigr)^n\to 
    \bigoplus_{I\in\{0,1,\infty\}^c}\CFa(\HD^I)
  \]
  be
  \[
    \mu_{1+n}(\cdot,X,\dots,X)=\sum_{I_0<\cdots<I_m}\mu_{1+n}^{I_0<\cdots<I_m}(\cdot,\overbrace{X,\cdots,X}^n).
  \]
  Note that in the special case $n=0$, $\mu_1$ is the map $D$
  introduced by Ozsv\'ath-Szab\'o, and in the case $n=1$,
  $\mu_2(\cdot,X)$ is the map $a^\zeta$ introduced by Hedden-Ni.

  We claim that the $\mu_{1+n}$ make
  $M=\bigoplus_{I\in\{0,1,\infty\}^c}\CFa(\HD^I)$ into an $\Ainf$-module. This follows by considering the ends of the $1$-dimensional
  moduli spaces of polygons with point constraints. With notation as in the rigid case discussed above, the ends of the $1$-dimensional moduli spaces are:
  \begin{itemize}[leftmargin=*]
  \item Ends where the polygon degenerates as a polygon on
    $(\alphas,\betas^0,\dots,\betas^{k})$ with $p$ point constraints
    and a polygon on $(\alphas,\betas^k,\dots,\betas^{m})$ with $n-p$
    point constraints. These degenerations correspond to the terms 
    \[
      \mu_{1+n-p}(\mu_{1+p}(\cdot,X,\dots,X),X,\dots,X)
    \]
    in the $\Ainf$ relation.  (This uses the fact that the pushoffs
    $A_i$ are chosen consistently, as in
    Section~\ref{sec:holomorphic}, so that the count of curves
    constrained by $A_1,\dots,A_{n-p}$ and the count of curves
    constrained by $A_{p+1},\dots,A_n$ are the same.)
  \item Ends where the polygon degenerates as a polygon on
    $(\alphas,\betas^0,\dots,\betas^i,\betas^j,\dots,\betas^k)$ with
    $n$ point constraints and a polygon on
    $(\betas^i,\betas^{i+1},\dots,\betas^k)$. These contributions
    vanish because the count of rigid polygons on
    $(\betas^i,\betas^{i+1},\dots,\betas^k)$ with corners at the
    $\Theta^i$ is zero~\cite[Lemma 4.5]{OSz-hf-dcov}.
  \item Ends where a pair of constrained points collide. These cancel
    in pairs, as in the proof of Lemma~\ref{HF-ainf-well-defd}
  \end{itemize}
  (Compare~\cite[Section 4.2]{OSz-hf-dcov}, \cite[Theorem 3.4 and
  Lemma 4.4]{HN-kh-detects}.)

  By construction, given a crossing $C_0$ of $L$, $M$ is the mapping cone of an $\Ainf$-module homomorphism
  \[
    \bigoplus_{\{I\in\{0,1,\infty\}^{c}\mid I(C_0)\in\{0,1\}\}}\CFa(\HD^I)\to
    \bigoplus_{\{I\in\{0,1,\infty\}^{c}\mid I(C_0)=\infty\}}\CFa(\HD^I),
  \]
  and the surgery exact triangle for $\HFa$ implies that this
  homomorphism is an isomorphism. So, it follows from the same
  inductive argument as in Ozsv\'ath-Szab\'o's case~\cite[Proof of
  Theorem 4.1]{OSz-hf-dcov} that 
  \[
    \wt{C}\coloneqq \bigoplus_{\{I\in\{0,1\}^{c}\}}\CFa(\HD^I)
  \]
  is quasi-isomorphic, as an $\Ainf$-module, to
  \[
    \CFa(\HD^{\infty,\infty,\dots,\infty})=\CFa(\Sigma(L)).
  \]

  Let
  \[
    C\coloneqq \bigoplus_{\{I\in\{0,1\}^{c}\}}\HFa(\HD^I).
  \]
  The complex $C$ is filtered by $|I|=\sum I$, the \emph{cube
    filtration}. That is,
  \[
    F^i=\bigoplus_{\{I\in\{0,1\}^{c}\mid |I|\geq i\}}\HFa(\HD^I).
  \]
  
  Choose a homotopy equivalence, over $\FF_2$, between each
  $\CFa(\HD^I)$ and $\HFa(\HD^I)$, so that the composition
  $\HFa(\HD^I)\to\CFa(\HD^I)\to\HFa(\HD^I)$ is the identity map.
  These homotopy equivalences induce maps
  $f\co C\to \wt{C}$ and $g\co \wt{C}\to C$ with $g\circ
  f=\Id_C$. Define the operation $\mu_1$ on $C$ by
  \[
    \mu_1(x)=g(\mu_1(f(x))).
  \]
  By homological perturbation theory (Proposition~\ref{prop:hpt}), $C$
  inherits the structure of an
  $\Ainf$-module over $\FF_2[X]/(X^2)$. (Equivalently, $C$ is obtained
  from $\wt{C}$ by canceling all differentials which do not change the
  cube filtration.) Further:
  \begin{itemize}[leftmargin=*]
  \item The operations $\mu_n$ on $C$ all respect the cube filtration.
  \item The differential $\mu_1$ on $C$ increases the cube filtration
    by at least $1$. Further, by construction, $\mu_1$ agrees with the
    differential on Ozsv\'ath-Szab\'o's cube~\cite[Proposition
    6.2]{OSz-hf-dcov}, and hence the first-order part of $\mu_1$
    agrees with the Khovanov differential.
  \item The zeroth-order part of $\mu_2$ is induced by the
    $H_1/\tors$-action on the groups
    $\HFa(\HD^I)=\HFa(\#^{k(I)}(S^2\times S^1))$. Hence, by
    Hedden-Ni~\cite[Theorem 4.5]{HN-kh-detects} (or inspection), the
    zeroth-order part of $\mu_2$ agrees with the $X$-action on Khovanov homology.
  \end{itemize}
  This completes the proof.
\end{proof}

\section{Proof of the detection theorems}\label{sec:detect}
We will use the following:
\begin{proposition}\label{prop:sphere-split}\cite[Proposition 5.1]{HN-hf-small}
  Let $L$ be a link in $S^3$. If $L=L_1\# L_2$ then
  $\Sigma(L)=\Sigma(L_1)\#\Sigma(L_2)$. If $L=L_1\amalg L_2$ then
  $\Sigma(L)=\Sigma(L_1)\#\Sigma(L_2)\# (S^2\times S^1)$. If $L$ is a
  non-split prime link then $\Sigma(L)$ is irreducible. If $L$ is a
  non-split link then $\Sigma(L)$ has no homologically essential
  $2$-spheres (i.e., no $S^2\times S^1$ summands).
\end{proposition}

\begin{corollary}\label{cor:which-spheres}
  Let $L$ be a link in $S^3$ and $p,q$ points in $L$. Let $\gamma$ be
  a path in $S^3$ from $p$ to $q$ with the interior of $\gamma$
  disjoint from $L$, and let $\zeta\subset \Sigma(L)$ be the preimage
  of $\gamma$. If there is an embedded sphere $\wt{S}\subset \Sigma(L)$ so that
  $\gamma\cdot \wt{S}$ is nonzero then there is an embedded sphere
  $S\subset S^3\setminus L$ separating $p$ and $q$.
\end{corollary}
\begin{proof}
  This follows from the same argument used to prove
  Proposition~\ref{prop:sphere-split}, but we can also deduce it from
  Proposition~\ref{prop:sphere-split}.

  We will prove the contrapositive.
  Assume there is no sphere separating $p$ and $q$.  Write
  $L=L_1\amalg\cdots\amalg L_k$ as a (split) disjoint union of links,
  so that each $L_i$ is non-split. Let $B_1,\dots,B_k$ be disjoint
  balls around $L_1,\dots,L_k$.

  Reordering the $L_i$, suppose that $p,q\in L_1$. As shown in
  Section~\ref{sec:ss-respect}, the homology class of $\zeta$ is
  independent of the choice of $\gamma$. So, we can assume that
  $\gamma$ is contained in $B_1$. By
  Proposition~\ref{prop:sphere-split} we have
  \[
    \Sigma(L)=\Sigma(L_1)\#\cdots\#\Sigma(L_k)\#(S^2\times S^1)^{\#(k-1)},
  \]
  and each $\Sigma(L_i)$ has no homologically essential $2$-spheres.
  The curve $\gamma$ lies in $\Sigma(L_1)$, so is disjoint from all of
  the homologically essential $2$-spheres. This proves the result.
\end{proof}

\subsection{Khovanov homology of split links}
\begin{lemma}\label{lem:split-is-free}
  Let $L$ be a link and $p,q$ points on $L$. Write $\rKhCx(L)_{p,q}$
  for the reduced Khovanov complex of $L$, reduced at $p$ and viewed
  as a module over $\FF_2[X]/(X^2)$ via the basepoint $q$. If there is
  a $2$-sphere in $S^3\setminus L$ separating $p$ and $q$ then
  $\rKh(L;\FF_2)$ is a free module.
\end{lemma}
\begin{proof}
  By Corollary~\ref{cor:Kh-invariance}, we may assume that we are
  computing $\rKh(L;\FF_2)$ from a split diagram, i.e., the disjoint
  union of a link diagram $L_1$ containing $p$ and a link diagram
  $L_2$ containing $q$. Then,
  $\rKh(L;\FF_2)\cong \rKh(L_1;\FF_2)\otimes\Kh(L_2;\FF_2)$ as
  $\FF_2[X]/(X^2)$-modules. By a result of
  Shumakovitch~\cite[Corollary 3.2.B]{Shu-kh-torsion},
  $\Kh(L_2;\FF_2)$ is a free module, so $\rKh(L;\FF_2)$ is, as well.
  %
\end{proof}

\subsection{Detection of split links by Khovanov homology}
Let $\Lambda$ denote the universal Novikov field over $\FF_2$,
consisting of formal sums $\sum_i f_i t^{r_i}$ where the $f_i\in \FF_2$,
$r_i\in\RR$, and $\lim_{i\to\infty} r_i=\infty$. An element
$\omega\in H^2(Y;\RR)$ induces a map $H_2(Y;\ZZ)\to \RR$ and hence a
ring homomorphism $\FF_2[H_2(Y;\ZZ)] \to \Lambda$. This makes $\Lambda$
into a module $\Lambda_\omega$ over $\FF_2[H_2(Y;\ZZ)]$.
\begin{citethm}\cite[Theorem 1.1]{AL-hf-incomp}\label{thm:AL-detect}
  Let $Y$ be a closed, oriented $3$-manifold and
  $\omega\in H^2(Y;\RR)$. Then, $\tHFa(Y;\Lambda_\omega)=0$ if and
  only if $Y$ contains a $2$-sphere $S$ so that $\int_S\omega\neq 0$.
\end{citethm}

\begin{corollary}\label{cor:detect}
  Suppose that $\omega\in \Hom(H_2(Y;\ZZ),\ZZ)$, and let $\FF_2[t^{-1},t]_\omega$ be
  $\FF_2[t^{-1},t]$, viewed as an $\FF_2[H_2(Y;\ZZ)]$-module via
  $\omega$. Then, $\tHFa(Y;\FF_2[t^{-1},t]_\omega)$ is a torsion
  $\FF_2[t^{-1},t]$-module if and only if $Y$ contains a $2$-sphere $S$
  so that $\omega([S])\neq 0$.
\end{corollary}
\begin{proof}
  In the case $\omega=0$, this is the well-known statement that
  $\HFa(Y)\neq 0$ (see~\cite[Theorem 1.2]{AL-hf-incomp}). So, assume
  $\omega\neq 0$.

  Let $\FF_2(t)$ denote the field of rational functions in $t$; this is
  also the field of fractions of $\FF_2[t,t^{-1}]$. The module
  $\tHFa(Y;\FF_2[t^{-1},t]_\omega)$ is torsion if and only if
  $\tHFa(Y;\FF_2[t^{-1},t]_\omega)\otimes_{\FF_2[t^{-1},t]}\FF_2(t)=0$. It follows from
  the universal coefficient theorem that
  \begin{align*}
    \tHFa(Y;\FF_2(t)_\omega)&\cong \tHFa(Y;\FF_2[t^{-1},t]_\omega)\otimes_{\FF_2[t^{-1},t]}\FF_2(t)\\
    \tHFa(Y;\Lambda_\omega)&\cong \tHFa(Y;\FF_2(t)_\omega)\otimes_{\FF_2(t)}\Lambda.
  \end{align*}
  (Compare~\cite[Formula (2.1)]{AL-hf-incomp}.) So, since $\Lambda$ and
  $\FF_2(t)$ are fields,
  \[
    \dim_\Lambda \tHFa(Y;\Lambda_\omega) = \dim_{\FF_2(t)}\tHFa(Y;\FF_2(t)_\omega).
  \]
  Hence, $\dim_\Lambda \tHFa(Y;\Lambda_\omega)=0$ if and only if
  $\tHFa(Y;\FF_2[t^{-1},t]_\omega)$ is torsion, so the result follows from
  Theorem~\ref{thm:AL-detect}.
\end{proof}

\begin{lemma}\label{lem:HF-nontriv}
  Let $\FF_2$ be any field.  If $Y$ is a 3-manifold with
  $H_1(Y)\cong \ZZ$ and no homologically essential 2-spheres in $Y$
  then the unrolled homology of $\CFa(Y;\FF_2)$ is nontrivial. More
  generally, for any $3$-manifold $Y$, if $\zeta\in H_1(Y)$ is such
  that the intersection number $\zeta\cdot S=0$ for all $2$-spheres
  $S\subset Y$ then the unrolled homology of $\CFa(Y;\FF_2)$ with
  respect to $\zeta$ is nontrivial.
\end{lemma}
\begin{proof}
  We prove the more general statement. Let
  $\omega\in \Hom(H_2(Y),\ZZ)$ be intersection with $\zeta$. By
  Corollary~\ref{cor:detect}, $\tHFa(Y;\FF_2[t^{-1},t]_\omega)$ has an
  $\FF_2[t^{-1},t]$-summand. So, by Corollary~\ref{cor:free-is-nontriv},
  the unrolled homology of $\CFa(Y;\FF_2)$ is nontrivial.
\end{proof}

\begin{corollary}\label{cor:HF-nontriv}
  If $L$ is a non-split, 2-component link then the unrolled homology
  of the complex $\CFa(\Sigma(L);\FF_2)$ is nontrivial.

  More generally, suppose $L=L_1\cup L_2$ is a union of two disjoint
  sublinks, $p\in L_1$, and $q\in L_2$. Endow $\CFa(\Sigma(L);\FF_2)$
  with the $\Ainf$-module structure over $\FF_2[X]/(X^2)$ coming from
  a lift of a path from $p$ to $q$ (\S\ref{sec:ss-respect}). If there
  is no $2$-sphere separating $L_1$ and $L_2$ then the unrolled
  homology of $\CFa(\Sigma(L);\FF_2)$ is nontrivial.
\end{corollary}
Here, if $\Sigma(L)$ is a rational homology sphere then we view
$\CFa(\Sigma(L))$ as a module over $\FF_2[X]/(X^2)$ trivially, i.e.,
$\mu_{1+n}(y,a_1,\dots,a_n)=0$ if $n>1$ or if $n=1$ and $a_1=X$.
\begin{proof}
  This is immediate from Lemma~\ref{lem:HF-nontriv} and
  Proposition~\ref{prop:sphere-split} (for the first statement) or
  Corollary~\ref{cor:which-spheres} (for the second statement).
\end{proof}

\begin{proof}[Proof of Theorem~\ref{thm:2-comp-precise}]
  \ref{item:L-split}$\implies$\ref{item:Kh-free} This is
  Lemma~\ref{lem:split-is-free}.

  \ref{item:L-split}$\implies$\ref{item:2-quasi-free} By
  Theorem~\ref{thm:Kh-mod-well-defd}, we may assume the diagram for
  $L$ is itself split. Then, the reduced Khovanov complex is itself a
  complex of free modules.

  \ref{item:Kh-free}$\implies$\ref{item:unroll-acyclic} This follows
  by considering the horizontal filtration on $\unroll{\KhCx}$. The
  $E^1$-page of the associated spectral sequence is the unrolled
  complex for $\rKh(L;\FF_2)$, and the unrolled complex of a free
  module is acyclic.

  \ref{item:2-quasi-free}$\implies$\ref{item:unroll-acyclic} This is
  immediate from Lemma~\ref{lem:qf-unroll-acyclic}.

  \ref{item:unroll-acyclic}$\implies$\ref{item:L-split} 
  Suppose that $L$ is a link and $p,q\in L$ are such that $\unroll{\rKhCx(L;\FF_2)}$ is
  acyclic.  Let $C_*$ be the complex from
  Proposition~\ref{prop:ss-resp-X}. By Lemma~\ref{lem:X-hom-well-def-ungr},
  the unrolled homology of $C_*$ is isomorphic to the unrolled
  homology of $\CFa(\Sigma(m(L));\FF_2)$. Let $\Filt$ be the sum of the cube
  filtration on $C_*$ and the horizontal filtration on
  $\unroll{C_*}$. (That is, for an element $x\in \unroll{C_*}$,
  $\Filt(x)$ is the sum of the horizontal filtration of $x$ and the
  minimal cube filtration of any term in $x$.)

  Consider the spectral sequence associated to the filtration $\Filt$.
  Since $\mu_1$ strictly raises the cube filtration, the
  $d_0$-differential vanishes. The differential on the $E^1$-page is
  the sum of the first-order part of $\mu_1$ and the differential on
  the unrolled complex coming from the zeroth order part of
  $\mu_2(\cdot,X)$. By Properties~\ref{item:mu1-first-order}
  and~\ref{item:mu2-zeroth-order} in Proposition~\ref{prop:ss-resp-X},
  this is exactly the differential on $\unroll{\rKhCx(m(L))}$. Hence,
  the $E^1$-page is acyclic. Since the complex $\unroll{C_*}$ is
  complete in the filtration $\Filt$, this implies that
  $\unroll{C_*}\simeq \unroll{\CFa(\Sigma(m(L));\FF_2)}$ is acyclic. Hence, by
  Corollary~\ref{cor:HF-nontriv}, there is a $2$-sphere in
  $S^3\setminus L$ separating $p$ an $q$.
\end{proof}

\begin{proof}[Proof of Theorem~\ref{thm:main-v1}]
  This is immediate from the equivalence of parts~\ref{item:L-split}
  and~\ref{item:Kh-free} in Theorem~\ref{thm:2-comp-precise}.
\end{proof}

\begin{proof}[Proof of Corollary~\ref{cor:unreduced}]
  Suppose there is a sphere in $S^3\setminus L$ separating $p$ and $q$. By
  Theorem~\ref{thm:Kh-mod-well-defd}, we may assume that $\Kh(L)$ is computed
  from a split diagram $L_p\amalg L_q$ with $p\in L_p$ and $q\in L_q$. Then,
  $\Kh(L)\cong \Kh(L_p)\otimes_{\FF_2}\Kh(L_q)$ as a module over
  $\FF_2[W,X]/(W^2,X^2)$. As in the proof of Lemma~\ref{lem:split-is-free},
  $\Kh(L_p)$ is a free module over $\FF_2[W]/(W^2)$ and $\Kh(L_q)$ is a free
  module over $\FF_2[X]/(X^2)$. So, $\Kh(L)$ is a free module over
  $\FF_2[W,X]/(W^2,X^2)$, as desired.

  Conversely, suppose that $\Kh(L)$ is a free module over
  $\FF_2[W,X]/(W^2,X^2)$. We claim that $\rKh(L)$ is a free module over
  $\FF_2[X]/(X^2)$. If we knew that all higher $\Ainf$ operations on $\Kh(L)$
  vanished then this would be immediate, since the reduced Khovanov homology is
  the $\Ainf$ tensor product of the $(\FF_2[W]/(W^2),\FF_2[X]/(X^2))$-bimodule
  $\Kh(L)$ over $\FF_2[W]/(W^2)$ with $\FF_2$ (Lemma~\ref{lem:reduce-Kh}). In
  fact, the result follows from homological algebra nonetheless. The $\Ainf$
  tensor product is
  \[
    \begin{tikzpicture}
      \node at (0,0) (zero) {$0$};
      \node at (2,0) (Kh1) {$\Kh(L)$};
      \node at (4,0) (Kh2) {$\Kh(L)$};
      \node at (6,0) (Kh3) {$\Kh(L)$};
      \node at (8,0) (Kh4) {$\Kh(L)$};
      \node at (10,0) (cdots) {$\cdots$};
      \draw[->] (Kh1) to (zero);
      \draw[->] (Kh2) to (Kh1);
      \draw[->] (Kh3) to (Kh2);
      \draw[->] (Kh4) to (Kh3);
      \draw[->] (cdots) to (Kh4);
      \draw[->, bend right=20] (Kh3) to (Kh1);
      \draw[->, bend right=20] (Kh4) to (Kh2);
      \draw[->, bend right=20] (cdots) to (Kh3);
      \draw[->, bend right=25] (Kh4) to (Kh1);
      \draw[->, bend right=25] (cdots) to (Kh2);
      \draw[->, bend right=30] (cdots) to (Kh1);
    \end{tikzpicture}
  \]
  where an arrow of length $n$ comes from the operation
  $m_{1+n}(\cdot,W,\dots,W)$. (More generally, an $\Ainf$-bimodule operation
  $\mu_{k,1,\ell}$ contributes an $\Ainf$-module operation
  $\mu_{1+\ell}$ which goes $k$ steps to the left.)

  Consider the spectral sequence
  associated to the obvious horizontal filtration. (This is a
  formulation of the universal coefficient spectral sequence.) Since
  $\Kh(L)$ is finitely generated and the $d^i$ differential changes
  the homological grading by $i-1$, the spectral sequence
  converges. The $E^2$-page is
  \[
    0\longleftarrow Kh(L)/\bigl(W\cdot \Kh(L)\bigr)\longleftarrow
    0\longleftarrow 0\longleftarrow 0\longleftarrow\cdots
  \]
  so the spectral sequence collapses. Thus, as an (ordinary) module,
  the $E^\infty$-page is $Kh(L)/\bigl(W\cdot \Kh(L)\bigr)$.  Further,
  the form of the $E^\infty$-page implies that the module structure on
  the $E^\infty$-page is the same as the module structure on the
  homology of the total complex.
  So, the reduced Khovanov homology is
  isomorphic to $\Kh(L)/ \bigl(W\cdot\Kh(L)\bigr)$, a free module over $\FF_2[X]/(X^2)$.
  Hence, by Theorem~\ref{thm:main-v1}, there is a sphere in $S^3\setminus L$
  separating $p$ and $q$.
\end{proof}

Finally, we note that Hedden-Ni's result that the Khovanov homology
module detects the unlink follows from Theorem~\ref{thm:main-v1}
(and~\cite{KM-kh-detects}). (Of course, the techniques we used to
prove Theorem~\ref{thm:main-v1} are similar to the ones they used.)
\begin{citethm}\cite[Theorem 2]{HN-kh-detects}
  Let $L$ be an $n$-component link and $U$ the $n$-component
  unlink. If
  $\Kh(L)\cong \Kh(U)=\FF_2[X_1,\dots,X_n]/(X_1^2,\dots,X_n^2)$, as
  modules over the ring $\FF_2[X_1,\dots,X_n]/(X_1^2,\dots,X_n^2)$, then
  $L\sim U$.
\end{citethm}
\begin{proof}
  By Corollary~\ref{cor:unreduced}, there is a sphere in $S^3\setminus L$
  separating each pair of components of $L$. It follows that $L$ is a disjoint
  union of $n$ knots. By the K\"unneth theorem for Khovanov homology, each of
  these knots has Khovanov homology $\FF_2[X]/(X^2)$. So, by Kronheimer-Mrowka's
  theorem~\cite{KM-kh-detects} that Khovanov homology detects the unknot, each
  component is an unknot, and so $L$ is an $n$-component unlink.
\end{proof}

\subsection{Detection of split links by Heegaard Floer homology}
\begin{proof}[Proof of Theorem~\ref{thm:HF-detect-split}]
  We will prove the properties stated in the theorem are equivalent to
  the following two additional properties, as well:
  \begin{enumerate}[label=(\ref{item:HF-free}\alph*),ref=\ref{item:HF-free}\alph*,leftmargin=*]
  \item\label{item:HF-unroll-acyclic} $\unroll{\HFa(\Sigma(L);\FF_2)}$
    is acyclic, where $\HFa(\Sigma(L);\FF_2)$ is viewed as an ordinary
    module over the ring $\FF_2[X]/(X^2)$.
  \item\label{item:HF-Ainf-unroll-acyclic}
    $\unroll{\HFa(\Sigma(L);\FF_2)}$ is acyclic, where
    $\HFa(\Sigma(L);\FF_2)$ is viewed as an $\Ainf$-module over the ring
    $\FF_2[X]/(X^2)$.
  \end{enumerate}
  The logic of the proof is:
  
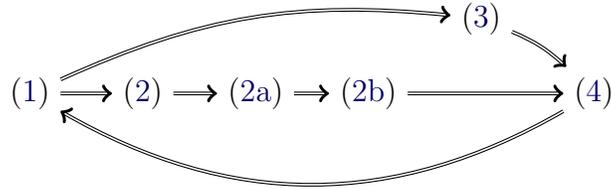
\[
    \begin{tikzpicture}
      \node at (0,0) (L-split) {(\ref{item:HF-L-split})};
      \node at (1.5,0) (HF-free) {(\ref{item:HF-free})};
      \node at (3,0) (HF-unroll-acyclic) {(\ref{item:HF-unroll-acyclic})};
      \node at (4.5,0) (HF-Ainf-unroll-acyclic) {(\ref{item:HF-Ainf-unroll-acyclic})};
      \node at (6,1) (HF-2-quasi-free) {(\ref{item:HF-2-quasi-free})};
      \node at (7.5,0) (HF-unroll-cx-acyclic) {(\ref{item:HF-unroll-cx-acyclic})};
      \draw[->, double] (L-split) to (HF-free);
      \draw[->, double] (HF-free) to (HF-unroll-acyclic);
      \draw[->, double] (HF-unroll-acyclic) to (HF-Ainf-unroll-acyclic);
      \draw[->, double] (HF-Ainf-unroll-acyclic) to (HF-unroll-cx-acyclic);
      \draw[->, double, bend left=30] (HF-unroll-cx-acyclic) to (L-split);
      \draw[->, double, bend left=15] (L-split) to (HF-2-quasi-free);
      \draw[->, double, bend left=10] (HF-2-quasi-free) to (HF-unroll-cx-acyclic);
    \end{tikzpicture}
  \]
  
  (\ref{item:HF-L-split})$\implies$(\ref{item:HF-free}) By
  Proposition~\ref{prop:sphere-split}, there is a decomposition
  $\Sigma(L)\cong \Sigma(L_1)\#\Sigma(L_2)\#(S^2\times S^1)$, where
  the loop $\zeta$ induced by $p$ and $q$ intersects the $2$-sphere
  $S^2\times\{\pt\}\subset S^2\times S^1$ algebraically once. So, the
  result follows from the K\"unneth theorem for $\HFa$ and a model
  computation of the $H_1/\tors$ action on
  $\HFa(S^2\times S^1)$~\cite{OSz-hf-3manifolds}.

  (\ref{item:HF-free})$\implies$(\ref{item:HF-unroll-acyclic}) This is
  immediate from the definition of the unrolled complex.

  (\ref{item:HF-unroll-acyclic})$\implies$(\ref{item:HF-Ainf-unroll-acyclic}) This
  follows from the spectral sequence associated to the horizontal
  filtration on $\unroll{\CFa(\Sigma(K))}$: the $E^1$-page is the
  unrolled complex for $\HFa(\Sigma(K))$ (viewed as an honest module
  over $\FF_2[X]/(X^2)$), which is acyclic by assumption.

  (\ref{item:HF-Ainf-unroll-acyclic})$\implies$(\ref{item:HF-unroll-cx-acyclic})
  This follows from invariance of the unrolled homology under $\Ainf$
  quasi-isomorphism (Lemma~\ref{lem:X-hom-well-def-ungr}) and the fact
  that any $\Ainf$-module is quasi-isomorphic to its homology.

  (\ref{item:HF-L-split})$\implies$(\ref{item:HF-2-quasi-free}) By
  Corollary~\ref{cor:which-spheres}, we can factor
  $\Sigma(L)=(S^2\times S^1)\# Y'$ where the loop $\zeta$
  induced by $p$ and $q$ is the circle
  $\{\pt\}\times S^1$ in the $S^2\times S^1$. Choose a Heegaard
  diagram $\HD=\HD_1\#\HD_2$ which witnesses this splitting, where
  $\HD_1$ is the standard Heegaard diagram for $S^2\times S^1$ and the
  connected sum happens in the region containing the basepoint
  $z$. Then, the chain complex $\CFa(\HD)$ is a free module over
  $\FF_2[X]/(X^2)$ (with trivial higher $\Ainf$-operations).

  (\ref{item:HF-2-quasi-free})$\implies$(\ref{item:HF-unroll-cx-acyclic})
  This is immediate from Lemma~\ref{lem:qf-unroll-acyclic}.
  
  (\ref{item:HF-unroll-cx-acyclic})$\implies$(\ref{item:HF-L-split})
  This is Corollary~\ref{cor:HF-nontriv}.
\end{proof}

\providecommand{\bysame}{\leavevmode\hbox to3em{\hrulefill}\thinspace}
\providecommand{\MR}{\relax\ifhmode\unskip\space\fi MR }
\providecommand{\MRhref}[2]{%
  \href{http://www.ams.org/mathscinet-getitem?mr=#1}{#2}
}
\providecommand{\href}[2]{#2}


\end{document}